\documentclass[arxiv=true]{agn_article}

\usepackage{agn_all}

\renewcommand{\dd}{\displaystyle}

\newcommand{\omegaSVK}{\omega_{\textrm{SVK}}}

\newcommand{\Wiso}{W_{\textrm{\rm iso}}}
\newcommand{\Wvol}{W_{\textrm{\rm vol}}}

\newcommand{\GLpz}{\GL^+(2)}

\newcommand{\quoteref}[1]{\enquote{\emph{#1}}}

\begin{document}

\title{A non-ellipticity result, or\\ the impossible taming of the logarithmic strain measure} 
\author{%
	Robert J.\ Martin\thanks{%
		Corresponding author: Robert J.\ Martin, \ \ Lehrstuhl f\"{u}r Nichtlineare Analysis und Modellierung, Fakult\"{a}t f\"{u}r Mathematik,
		Universit\"{a}t Duisburg-Essen, Thea-Leymann Str. 9, 45127 Essen, Germany; email: robert.martin@uni-due.de%
	}
	\quad and\quad
	Ionel-Dumitrel Ghiba\thanks{%
		Ionel-Dumitrel Ghiba, \ \ Alexandru Ioan Cuza University of Ia\c si, Department of Mathematics, Blvd.~Carol I, no.~11, 700506 Ia\c si, Romania;
		Octav Mayer Institute of Mathematics of the Romanian Academy, Ia\c si Branch, 700505 Ia\c si and Lehrstuhl f\"{u}r Nichtlineare Analysis und Modellierung, Fakult\"{a}t f\"{u}r Mathematik, Universit\"{a}t Duisburg-Essen, Thea-Leymann Str. 9, 45127 Essen, Germany; email: dumitrel.ghiba@uni-due.de, dumitrel.ghiba@uaic.ro%
	}
	\quad and\quad
	Patrizio Neff\thanks{%
		Patrizio Neff, \ \ Head of Lehrstuhl f\"{u}r Nichtlineare Analysis und Modellierung, Fakult\"{a}t f\"{u}r Mathematik, Universit\"{a}t Duisburg-Essen, Thea-Leymann Str. 9, 45127 Essen, Germany; email: patrizio.neff@uni-due.de%
	}
}
\maketitle
\begin{abstract}
The logarithmic strain measures $\norm{\log U}^2$, where $\log U$ is the principal matrix logarithm of the stretch tensor $U=\sqrt{F^TF}$ corresponding to the deformation gradient $F$ and $\norm{\,.\,}$ denotes the Frobenius matrix norm, arises naturally via the geodesic distance of $F$ to the special orthogonal group $\SO(n)$. This purely geometric characterization of this strain measure suggests that a viable constitutive law of nonlinear elasticity may be derived from an elastic energy potential which depends solely on this intrinsic property of the deformation, i.e.\ that an energy function $W\col\GLpn\to\R$ of the form
\begin{equation}
	W(F)=\Psi(\norm{\log U}^2) \label{eq:abstractFunctionForm}\tag{1}
\end{equation}
with a suitable function $\Psi\col[0,\infty)\to\R$ should be used to describe finite elastic deformations. %

However, while such energy functions enjoy a number of favorable properties, we show that it is not possible to find a strictly monotone function $\Psi$ such that $W$ of the form \eqref{eq:abstractFunctionForm} is Legendre-Hadamard elliptic.

Similarly, we consider the related isochoric strain measure $\norm{\dev_n\log U}^2$, where $\dev_n \log U$ is the deviatoric part of $\log U$. Although a polyconvex energy function in terms of this strain measure has recently been constructed in the planar case $n=2$, we show that for $n\geq3$, no strictly monotone function $\Psi\col[0,\infty)\to\R$ exists such that $F\mapsto \Psi(\norm{\dev_n\log U}^2)$ is polyconvex or even rank-one convex.
Moreover, a volumetric-isochorically decoupled energy of the form $F\mapsto \Psi(\norm{\dev_n\log U}^2) + \Wvol(\det F)$ cannot be rank-one convex for any function $\Wvol\col(0,\infty)\to\R$ if $\Psi$ is strictly monotone.
\end{abstract}

\tableofcontents

\section{Introduction}
\subsection{Strain measures in nonlinear elasticity}
\label{section:introduction}

In nonlinear hyperelasticity, the behaviour of an elastic material is determined by an \emph{elastic energy potential}\footnote{For the notation employed here and throughout, see Section \ref{notationsect}.}
\begin{equation}
	W\col\GLpn\to\R\,,\quad F\mapsto W(F)
\end{equation}
depending on the deformation gradient $F=\grad\varphi$ of a deformation $\varphi$. A large variety of representation formulae for certain classes of such functions is available in the literature. In particular, it is well known that any objective and isotropic function $W$ can be expressed in terms of the singular values of the argument, i.e.\ for any function $W\col\GLpn\to\R$ with
\[
	W(F)={W}(Q_1\, F\, Q_2) \qquad\text{ for all }\; Q_1,Q_2,R\in\SO(n)\,,
\]
there exists a unique symmetric function $g\col\mathbb{R}_+^n\to \mathbb{R}$ such that $W(F)=g(\lambda_1,\dotsc,\lambda_n)$ for all $F\in\GLpn$ with singular values $\lambda_1,\dotsc,\lambda_n$. Furthermore, if $f\col[0,\infty)\to\R$ is injective, then $W$ can also be written as
\begin{align}
W(F)=\widetilde{g}(f(\lambda_1),f(\lambda_2),...,f(\lambda_n))
\end{align}
with a symmetric function $\widetilde{g}\col\mathbb{R}\to \mathbb{R}$. In particular, this representation is possible for functions $f=f_{(m)}$ of the form
\begin{align}
f_{(m)}(x)=\left\{
\begin{array}{ll}
\dd\frac{1}{2\.m}(x^{2m}-\id),& m\neq 0\,, \vspace{1.2mm}\\
\log x,& m=0\,,
\end{array}
\right.
\end{align}
which correspond to the commonly used \cite{boehlkeBertram2002} \emph{strain tensors} of \emph{Seth-Hill type} \cite{seth1961generalized,hill1970}
\begin{align}
E_{(m)}=\left\{
\begin{array}{ll}
\dd\frac{1}{2\, m}(U^{2\, m}-\id),& m\neq 0, \vspace{1.2mm}\\
\log U,& m=0\,,
\end{array}
\right.
\end{align}
where $\log U$ is the principal matrix logarithm of the stretch tensor $U=\sqrt{F^TF}$. In general, a (material) strain tensor is commonly defined as a \quoteref{uniquely invertible isotropic second order tensor function} of the right Cauchy-Green deformation tensor $C=F^TF$ \cite[p.~268]{truesdell60}.\footnote{Additional properties such as monotonicity are sometimes required of strain tensors, see for example \cite[p.\ 230]{Hill68} or \cite[p.~118]{Ogden83}.} %
Due to the invertibility of strain tensor mappings, any energy function $W$ can also be written in the form
\[
	W(F) = \widetilde{W}(E(F))
\]
for any strain tensor $E$.

In contrast to a strain tensor, a \emph{strain measure} is an arbitrary mapping $\omega\col\GLp(n)\to\R$ such that $\omega(F)=0$ if and only if $F\in\SOn$. Examples of strain measures include the squared Frobenius norms of the Seth-Hill strain tensors
\begin{equation}
	\omega_{(m)} = \norm{E_{(m)}}^2=\sum_{i=1}^n f_{(m)}^2(\lambda_i)\,.
\end{equation} 
From this perspective, a strain measure indicates how much a deformation gradient $F\in\GLp(n)$ differs from a pure rotation, which suggests that an appropriate strain measure should be defined by introducing a \emph{distance} function on $\GLpn$ and using the distance of $F$ to the space of pure rotations $\SO(n)$, or a function thereof, as the strain measure.

The particular choice of a suitable distance on $\GLpn$, however, is not immediately obvious. Grioli \cite{Grioli40,agn_neff2014grioli} showed that employing the Euclidean distance on $\GLp(n)$ yields the strain measure
\begin{equation}
	\omega_{\rm Grioli}\colonequals{\rm dist}_{\rm Euclid }^2 (F, \SO(n ))=\norm{U-\id}^2=\norm{E_{(1)}}^2=\omega_{(1)}\,.
\end{equation}
However, this strain measure suffers from a number of serious shortcomings due to the fact that the Euclidean distance is not an intrinsic distance measure on $\GLp(n)$ \cite{agn_neff2014riemannian,agn_neff2015geometry,agn_martin2014minimal}. On the other hand, strain measures involving the logarithmic strain arise from choosing the \emph{geodesic distance} on $\GLpn$ endowed with a natural Riemanian metric structure. More precisely \cite{agn_neff2015geometry},
\begin{align}\label{gsm}
\norm{\log U}^2&={\rm dist}_{{\rm geod}}^2(F, \SO(n)),\notag\\
\norm{\dev_n\log U}^2&={\rm dist}^2_{{\rm geod,{\rm SL}(n)}}\left( \frac{F}{(\det F)^{1/n}}, \SO(n)\right),\\
[\tr(\log U)]^2&=[\log \det U]^2={\rm dist}^2_{{\rm geod,\mathbb{R}_+\cdot \id}}\left((\det F)^{1/n}\cdot \id,
\id\right),\notag
\end{align}
where ${\rm dist}_{{\rm geod}}$, ${\rm dist}_{{\rm geod,\mathbb{R}_+\cdot \id}}$ and ${\rm dist}_{{\rm geod,{\rm SL}(n)}}$ are the canonical left invariant geodesic distances on the Lie-groups $\GLn$, ${\rm SL}(n)\colonequals\{X\in \GL(n)\;|\det{X}=1\}$ and $\mathbb{R}_+\cdot\id$, respectively \cite{agn_neff2015geometry,agn_lankeit2014minimization}.

Energy functions and constitutive laws expressed in terms of these logarithmic strain measures have been a subject of interest in nonlinear elasticity theory for a long time, going back to investigations by the geologist G.\,F.~Becker \cite{becker1893,agn_neff2014rediscovering} first published in 1893 and the famous introduction of the quadratic Hencky strain energy
\[
	\WH\col\GLpn\to\R\,,\quad \WH(F) = \mu\,\norm{\dev_n\log U}^2+\frac{\kappa}{2}\,[\tr(\log U)]^2 = \mu\,\norm{\log U}^2+\frac{\Lambda}{2}\,[\tr(\log U)]^2
\]
by Heinrich Hencky in 1929 \cite{Hencky1928,Hencky1929}. Hencky later considered more general elastic energy functions based on the logarithmic strain as well; for example, in a 1931 article in the Journal of Rheology \cite{hencky1931}, he suggested an energy function of the form
\begin{equation}
\label{eq:hencky1931}
	W_{1931}(F) = \mu\,\norm{\dev_3\,\log U}^2 + h(\det U)\,,
\end{equation}
were the volumetric part $h\col(0,\infty)\to\R$ of the energy is a function to be determined by experiments.

Important contributions are also due to H.~Richter, who considered the three logarithmic invariants $K_1=\tr(\log U)$, $K_2^2= \tr((\dev_3\log U)^2)$ and $\widetilde{K}_3=\tr((\dev_3\log U)^3)$ in a 1949 article \cite{richter1949verzerrung}.

More recently, a set of isotropic invariants similar to those used by Richter were introduced by Criscione et al.\ \cite{criscione2000invariant,criscione2002direct,wilber2005baker,idjeri2011identification}, who considered the \emph{invariant basis}
\begin{equation}\label{gsm2}
	\left\{ \begin{alignedat}{2}
		&K_1=\tr(\log U)=\log \det U &&\text{\enquote{the amount-of-dilatation}}\\[.63em]
		&K_2=\norm{\dev_3\log U} &&\text{\enquote{the magnitude-of-distortion}}\\
		&K_3=3\sqrt{6}\, \det \left(\dd\frac{\dev_3 \log U}{\norm{\dev_3\log U}}\right) \qquad\quad &&\text{\enquote{the mode-of-distortion}}
	\end{alignedat} \right.
\end{equation}
for the natural strain $\log U$ and showed that any isotropic energy $W$ on $\GLp(3)$ can be represented in the form
\begin{equation}\label{eq:criscioneEnergy}
	W(F)=W_{\rm {C}risc}(K_1,K_2,K_3)\,.
\end{equation}
Similarly, Lurie \cite{lurie2012nonlinear} used the invariants $K_1$, $K_2$ and $\widehat{K}_3=\arcsin (K_3)$.

Although energy functions expressed in terms of logarithmic strain measures often exhibit some interesting and desirable properties \cite{Anand79,agn_neff2015exponentiatedI,agn_neff2015exponentiatedII}, they also provide a number of mathematical challenges. One of the greatest difficulties is posed by the lack of appropriate \emph{convexity properties}.

\subsection{Convexity properties of energy functions}

Among the many constitutive properties for hyperelasticity discussed in the literature, some of the most important ones are the conditions of \emph{rank-one convexity} and \emph{polyconvexity} \cite{ball1977constitutive} of the energy function $W$.

\begin{definition}
	\label{definition:rankOneConvexity}
	A function $W\colon\GLpn\to\mathbb{R}$ is called \emph{rank-one convex} if for all $F\in\GLpn$, all $\xi,\eta\in\mathbb{R}^n$ and any interval $I\subset\R$ such that $\det (F+t\cdot \xi\otimes \eta)>0$ for all $t\in I$, the mapping
	\[
		I\to\R\,,\quad t\mapsto W(F+t\cdot \xi\otimes \eta)
	\]
	is convex.
\end{definition}

\begin{remark}
	A sufficiently regular function $W$ is rank-one convex on $\GLp(n)$ if and only if it satisfies the \emph{Legendre-Hadamard ellipticity} condition
	\begin{align}
	\label{def:lhellipt}
		D^2_F W(F)(\xi\otimes\eta,\xi\otimes\eta) \geq 0 \qquad\text{for all }\; \xi,
		\eta\in\mathbb{R}^n\,,\; F\in \GLp(n)\,.
	\end{align}
	If strict inequality holds in \eqref{def:lhellipt} for all $F\in\GLpn$ and all $\xi,\eta\in\R^n\setminus\{0\}$, then $W$ is called \emph{strongly} Legendre-Hadamard elliptic.
\end{remark}

\begin{definition}\label{definition:polyconvexityGLSL}
	~
	\begin{itemize}
		\item[i)]
			A function $W\colon\Rnn\to\R\cup\{\infty\}$ is called polyconvex if there exists a convex function $P\colon\mathbb{R}^m\to\mathbb{R}\cup\{\infty\}$ such that
			\begin{equation}
				W(F) = P(\mathbb{M}(F)) \qquad\text{for all }\;F\in\mathbb{R}^{n\times n}\,,
			\end{equation}
			where $\mathbb{M}(F)\in\R^m$ denotes the vector of all minors of $F$.
		\item[ii)]
			A function $W\colon\GLpn\to\R$ is called polyconvex if the function
			\begin{align}
				\widetilde{W}\colon\Rnn\to\R\cup\{\infty\}\,,\quad \widetilde{W}(F)=
				\begin{cases}
					W(F) &:\; F\in\GLpn\\
					\infty &:\; F\notin\GLpn
				\end{cases}
			\end{align}
			is polyconvex according to i).
	\end{itemize}
\end{definition}

\begin{remark}
\label{remark:polyImpliesRankOne}
	If $W\colon\GLpn\to\R$ is polyconvex, then $W$ is rank-one convex \cite{Dacorogna08}.
\end{remark}

Although, unlike many other constitutive assumptions, the condition of polyconvexity is not necessitated by physical or mechanical considerations, it is one of the most important tools to ensure the existence of energy minimizers under appropriate boundary conditions.

Rank-one convexity (or LH-ellipticity), on the other hand, is generally not sufficient to ensure the existence of minimizers. However, it is not only a necessary condition for polyconvexity \cite{Dacorogna08,wilber2002,ndanou2014criterion}, but directly motivated by physical reasoning as well: for example, ellipticity of a constitutive law ensures finite wave propagation speed \cite{eremeyev2007constitutive,zubov2011,sawyersRivlin78} and prevents discontinuities of the strain along plane interfaces under homogeneous Cauchy stress \cite{agn_neff2016injectivity,agn_mihai2016hyperelastic,agn_mihai2017hyperelastic}.

However, constructing a viable energy function in terms of logarithmic strain measures which satisfies either of these convexity conditions turns out to be quite challenging.

In a 2004 article, Sendova and Walton \cite{sendova2005strong} gave a number of necessary conditions for the rank-one convexity of energies of the form \eqref{eq:criscioneEnergy}. In the planar case $n=2$, it was recently shown by Neff et al.\ \cite{agn_neff2015exponentiatedII,agn_ghiba2015exponentiated} that the \emph{exponentiated Hencky energy}
\begin{equation}\label{eq:expHenckyIntroduction}
	W_{_{\rm eH}}\col\GLp(2)\to\R\,,\quad W_{_{\rm eH}}(F) = \frac{\mu}{k}\,e^{k\,\norm{\dev_2\log U}^2}+\frac{\kappa}{2\,\widehat{k}}\,e^{\widehat{k}\,[(\log \det U)]^2}\,,
\end{equation}
where $k\geq\frac14$ and $\widehat{k}\geq\frac18$ are additional dimensionless parameters, is polyconvex (and thus quasiconvex and rank-one convex). In the three-dimensional case, however, the exponentiated Hencky energy is \emph{not} rank-one convex.

As we will show in this article, the search for a rank-one convex energy resembling \eqref{eq:expHenckyIntroduction} in the three-dimensional case was, unfortunately, destined to fail from the beginning: For $n\geq3$, there exists no strictly monotone function of $\norm{\log U}^2$ or $\norm{\dev_n\log U}^2$ which is rank-one convex on $\GLpn$; further, an energy with a volumetric-isochoric split whose isochoric part is a strictly monotone function of $\norm{\dev_n\log U}^2$ cannot be rank-one convex. These main results are presented in Sections \ref{log}, \ref{devlog} and \ref{volisosect}, respectively. %

\subsection{Related work}
Bertram et al.\ \cite{bertram2007rank} considered quadratic energies of the form
\begin{align}
W(F)=g(\lambda_1,\lambda_2,...,\lambda_n)=\frac{1}{2}\sum_{i=1}^n f^2(\lambda_i)+\beta\, \sum_{1\leq i<j\leq n}f(\lambda_i)f(\lambda_j)\,,
\end{align}
with $f\col\mathbb{R}_+\to \mathbb{R}$ such that $f(1)=0,$ $f'(1)=0$, $f'\neq 0$ and $\beta\in \mathbb{R}$. The functions $f$ are known as \emph{generalized strain measures} \cite{boehlkeBertram2002}. The authors prove that if the Hessian
of $g$ at $(1,1,1)$ is positive definite, $\beta\neq 0$, and $f$ is strictly monotone, and/or if $f^2$ is a Seth-Hill strain measure $\omega_{(m)}$ corresponding to any $m\in \mathbb{R}$, then the energy $W$ is not rank-one convex. This extends previous results \cite{Raoult1986,Neff_Diss00,Bruhns01} for $f=f_{(1)}$ and $f=f_{(0)}$.
From these observations, the authors conclude that a necessary condition for an energy to be rank-one convex is that the stress-strain relationship in the considered generalized strain measures must be physically non-linear.

\subsection{Notation}\label{notationsect}
Throughout this article, $F=\grad\varphi$ denotes the deformation gradient corresponding to a deformation $\varphi$, $C=U^TU$ is the right Cauchy-Green deformation tensor, $B=FF^T$ is the Finger tensor, $U=\sqrt{F^TF}$ is the right stretch tensor and $V=\sqrt{FF^T}$ is the left stretch tensor corresponding to $F$.

Furthermore, we denote the standard Euclidean scalar product on $\R^{n\times n}$ by $\langle {X},{Y}\rangle=\tr{(X Y^T)}$, the Frobenius tensor norm is given by $\norm{{X}}^2=\langle {X},{X}\rangle$ and the identity tensor on $\R^{n\times n}$ is denoted by $\id$; note that $\tr{(X)}=\langle {X},{\id}\rangle$. We adopt the usual abbreviations of Lie-group theory, i.e.\ $\GL(n)\colonequals\{X\in\R^{n\times n}\;|\det{X}\neq 0\}$ denotes the general linear group, $\OO(n)\colonequals\{X\in \GL(n)\;|\;X^TX=\id\}$ is the orthogonal group, $\SO(n)\colonequals\{X\in \GL(n,\R)\;| X^T X=\id,\;\det{X}=1\}$ is the special orthogonal group and $\GLp(n)\colonequals\{X\in\R^{n\times n}\;|\det{X}>0\}$ is the group of invertible matrices with positive determinant. The superscript $^T$ is used to denote transposition, and $\Cof A = (\det A)A^{-T}$ is the cofactor of $A\in \GLp(n)$. For all vectors $\xi,\eta\in\R^n$ we denote the dyadic product by $(\xi\otimes\eta)_{ij}\colonequals\xi_i\,\eta_j$. By \enquote{$\cdot$} we denote the multiplication with scalars or the multiplication of matrices. The Fr\'echet derivative of a function $W$ at $F$ applied to the tensor-valued increment $H$ is denoted by $D_F[W(F)]. H$. Similarly, $D_F^2[W(F)]. (H_1,H_2)$ is the bilinear form induced by the second Fr\'echet derivative of the function $W$ at $F$ applied to $(H_1,H_2)$.

We also identify the first derivative $D_F W$ with the gradient, writing $D_F W.H = \iprod{D_F W, H}$ for $F\in\GLpn$ and $H\in\Rnn$, and employ the chain rules
\begin{align}
D_F(( \Phi\circ W)(F)).H=D_F(\Phi(W(F))).H= \Phi'(W(F))\, D_FW(F).H\,,\notag\\
D_F((W\circ G)(F)).H=D_F(W(G(F))).H= \langle DW(G(F)), D_FG(F).H\rangle\notag
\end{align}
for $W\col\mathbb{R}^{3\times3}\to \mathbb{R}$, $G\col\mathbb{R}^{3\times3}\to \mathbb{R}^{3\times3}$ and $\Phi\col\mathbb{R}\to \mathbb{R}$. For instance,
\begin{align}\label{fdsvk}
D_F(\norm{F^TF-\id}^2). H&=2\,\langle F^TF-\id, D_F(F^TF-\id).H\rangle=2\,\langle F^TF-\id, D_F(F^TF-\id).H\rangle\notag\\
&=4\,\langle F^TF-\id, F^TH+H^TF\rangle=4\,\langle FF^TF-F, H\rangle\,,
\end{align}
since
\begin{align}
D_F(\norm{F}). H=\frac{1}{\norm{F}}\,\langle F, H\rangle,\qquad D_F(\norm{F}^2).H=2\, \norm{F}\, \frac{1}{\norm{F}}\,\langle F, H\rangle=2\,\langle F, H\rangle \notag
\end{align}
for all $F\neq 0$.

\section{Rank-one convex energies in terms of strain measures}
We consider the problem of rank-one convexity in terms of different strain measures $\omega$. More specifically, we are interested in whether or not it is \emph{possible} for a given $\omega$ to find a non-trivial (i.e.\ non-constant) function $\Psi$ such that $\Psi\circ\omega$ is rank-one convex. A basic necessary condition on $\omega$ for the existence of a \emph{strictly monotone} function $\Psi\col\R\to\R$ such that the mapping $F\mapsto\Psi(\omega(F))$ is rank-one convex is stated in the following lemma.

\begin{lemma}
\label{lemma:impossibility}
	Let $\omega\in {\rm C}^2(\GLpn; I)$ for an interval $I\subset\R$. If there exist $F\in\GLpn$ and $\xi,\eta\in\mathbb{R}^n\setminus \{0\}$ such that
	\begin{equation}\label{eq:omegaCondition}
		D\omega(F).(\xi\otimes \eta) = 0 \qquad\text{and}\qquad D^2\omega(F).(\xi\otimes \eta,\xi\otimes \eta) < 0\,,
	\end{equation}
	then there exists no strictly monotone function $\Psi\col I\to\R$ such that the mapping $F\mapsto W(F)\colonequals\Psi(\omega(F))$ is rank-one convex on $\GLpn$.
\end{lemma}

\begin{proof}
	Let $F\in\GLpn$ and $\xi\otimes \eta\in\Rnn$ satisfy \eqref{eq:omegaCondition}. Then for $\eps>0$ sufficiently small, the mapping
	\[
		p\col (-\eps,\eps)\to I\,,\quad p(t) = \omega(F+t\cdot \xi\otimes \eta)
	\]
	has a strict maximum at $t=0$, since
	\[
		p'(0) = D\omega(F).(\xi\otimes \eta) = 0 \qquad\text{and}\qquad p''(0) = D^2\omega(F).(\xi\otimes \eta,\xi\otimes \eta) < 0\,.
	\]
	If $\Psi$ is strictly monotone on $I$, then the mapping
	\begin{equation}\label{eq:impossibilityLemmaNonConvexity}
		q\col (-\eps,\eps)\to \R\,,\quad q(t) = \Psi(p(t)) = W(F+t\cdot \xi\otimes \eta)
	\end{equation}has a strict maximum in $t=0$ as well. In particular, $q$ cannot be convex, which implies that $W$ is not rank-one convex (cf.\ Definition \eqref{definition:rankOneConvexity}).
\end{proof}
If $\Psi$ is twice differentiable on $I$ with $\Psi'(t)>0$ for all $t\in I$, then Lemma \ref{lemma:impossibility} also follows from the observation that
\begin{equation}
	D^2 W(F).(\xi\otimes \eta,\xi\otimes \eta) = \Psi''(\omega(F))\cdot \underbrace{[D \omega(F).(\xi\otimes \eta)]^2}_{=0}
		+ \underbrace{\Psi'(\omega(F))}_{>0}\cdot \underbrace{D^2 \omega(F).(\xi\otimes \eta,\xi\otimes \eta)}_{<0} \;<\; 0
\end{equation}
for $F\in\GLpn$ and $\xi\otimes \eta\in\Rnn$ satisfying \eqref{eq:omegaCondition}, since in that case, the Legendre-Hadamard ellipticity condition is violated at $F$.

\begin{remark}
	Note that by the usual interpretation of $\omega$ as the \emph{amount of strain} in a deformation, the assumption of (strict) monotonicity of $\Psi$ follows from basic physical reasoning, since an elastic energy function $W$ should always increase with increasing strain (cf.\ \cite[Section 2.2]{agn_neff2015exponentiatedI}). For example, if $\omega(F)=\norm{\log U}^2$, then the monotonicity of $\Psi$ is equivalent to the monotonicity of the mapping $t\mapsto W(t\cdot\id)$ on $(1,\infty)$, which, in turn, follows from the physically motivated requirement that the hydrostatic pressure corresponding to a purely volumetric deformation should be negative for extensions (and positive for compression).
	
	Furthermore, as we will discuss in Section \ref{Logarsectmon}, if $\omega$ is given by the deviatoric quadratic Hencky strain measure $\norm{\dev_3\log U}^2$ or by $\norm{\log U}^2$, then $\Psi'\geq 0$ \emph{must} hold everywhere if $\Psi\circ\omega$ is to be elliptic. In the deviatoric case, the strict inequality $\Psi'> 0$ also follows from the additional assumption that $W$ is compatible with linear elasticity, see Lemma \ref{remark:strictMonotonicityConditionLinearCompatibilityDev}.
\end{remark}

\subsection{A Saint-Venant-Kirchhoff type strain measure}
Before we apply Lemma \ref{lemma:impossibility} to the logarithmic strain measures discussed in Subsection \ref{section:introduction}, we consider the simple example of the \emph{Saint-Venant-Kirchhoff type strain measure}\footnote{Note that the classical Saint-Venant-Kirchhoff energy is well-known to be neither polyconvex nor rank-one convex \cite{Raoult1986}.}
\[
	\omegaSVK\col\GLpn\to\R\,,\quad \omegaSVK(F) = \norm{F^TF-\id}^2
\]
for arbitrary dimension $n\geq2$. Using \eqref{fdsvk}, we find 
\begin{align*}
	D\omegaSVK(F).H &= 4\.\iprod{FF^TF-F, H}\,,\\
	D^2\omegaSVK(F).(H,H) &= 4\.(\norm{HF^T}^2 + \norm{F^TH}^2 + \tr((F^TH)^2) - \norm{H}^2)
\end{align*}
for $F\in\GLpn$ and $H\in\Rnn$.
Thus, for $F=\frac12\.\id$ and the rank-one direction $H=e_1\otimes e_2$, where $e_i\in\R^n$ denotes the $i$-th unit vector, we find $\tr H=0$ and thus
\begin{align*}
	D\omegaSVK(F).H &= -\frac{3}{2}\.\iprod{\id, e_1\otimes e_2} = 0\,,\\
	D^2\omegaSVK(F).(H,H) &= 4\.(\norm{\tfrac12\.H}^2 + \norm{\tfrac12\.H}^2 + \tr((\tfrac12\.H)^2) - \norm{H}^2) = -\tfrac12\.\norm{H}^2 = -\tfrac12 < 0\,.
\end{align*}
Therefore, according to Lemma \ref{lemma:impossibility}, there is no strictly monotone increasing $\Psi\col[0,\infty)\to\R$ such that the mapping $F\mapsto\Psi(\norm{F^TF-\id}^2)$ is rank-one convex. In other words: for dimension $n\geq2$, there is no (physically viable) Legendre-Hadamard elliptic elastic energy in terms of the strain measure $\norm{F^TF-\id}^2 = \norm{C-\id}^2$.

\section{Logarithmic strain measures} \label{Logarsect}

Returning to the question of ellipticity of energy functions in terms of logarithmic strain measures, we start in the one-dimensional case. Identifying $\GLp(1)$ with $(0,\infty)$, the logarithmic strain measure $\norm{\log U}^2$ can be written as $(\log t)^2$ for $F=t\in(0,\infty)$.

\subsection{The one-dimensional case}

It is easily seen that the function $t\mapsto (\log t)^2$ is not convex. However, it is possible to find some function $\Psi\col\mathbb{R}_+\to \mathbb{R}$ which \enquote{convexifies} the logarithm in the sense that $t\mapsto \Psi((\log t)^2)$ is convex: Consider
\begin{align}\label{1d2d}
W(t)&=\Psi((\log t)^2),\notag\\
W^\prime(t)&=\Psi^\prime((\log t)^2)\, 2\, (\log t)\, \frac{1}{t},\\
W^{\prime\prime}(t)&=\Psi^{\prime\prime}((\log t)^2)\, 4\, (\log t)^2\, \frac{1}{t^2}+
\Psi^\prime((\log t)^2)\, 2\, \frac{1}{t^2}-\Psi^\prime((\log t)^2)\, 2\, (\log t)\, \frac{1}{t^2}\notag\\
&=\frac{1}{t^2}\left[\Psi^{\prime\prime}((\log t)^2)\, 4\, (\log t)^2\, +
\Psi^\prime((\log t)^2)\, 2\, (1- \log t)\right].\notag
\end{align}
Hence, the question whether $t\mapsto W(t)$ is convex can be restated as
\begin{align}
W^{\prime\prime}(t)\geq 0\quad &\iff\quad 
\Psi^{\prime\prime}((\log t)^2)\, 4\, (\log t)^2\, \geq-
\Psi^\prime((\log t)^2)\, 2\, (1- \log t)\notag\\ &\iff\quad 
\Psi^{\prime\prime}((\log t)^2)\geq -\frac{\frac{d^2}{dt^2}((\log t)^2)}{\left[\frac{d}{dt}((\log t)^2)\right]^2}=\frac{ \log t-1}{ 2\, (\log t)^2\,}\, \Psi^\prime((\log t)^2) \label{eq:1dfullConvexityStatement}
\end{align}
for all $t>0$ with $t\neq 1$.

\begin{figure}[h!]
	\begin{minipage}[h]{1\linewidth}
		\centering
		\includegraphics[scale=0.7]{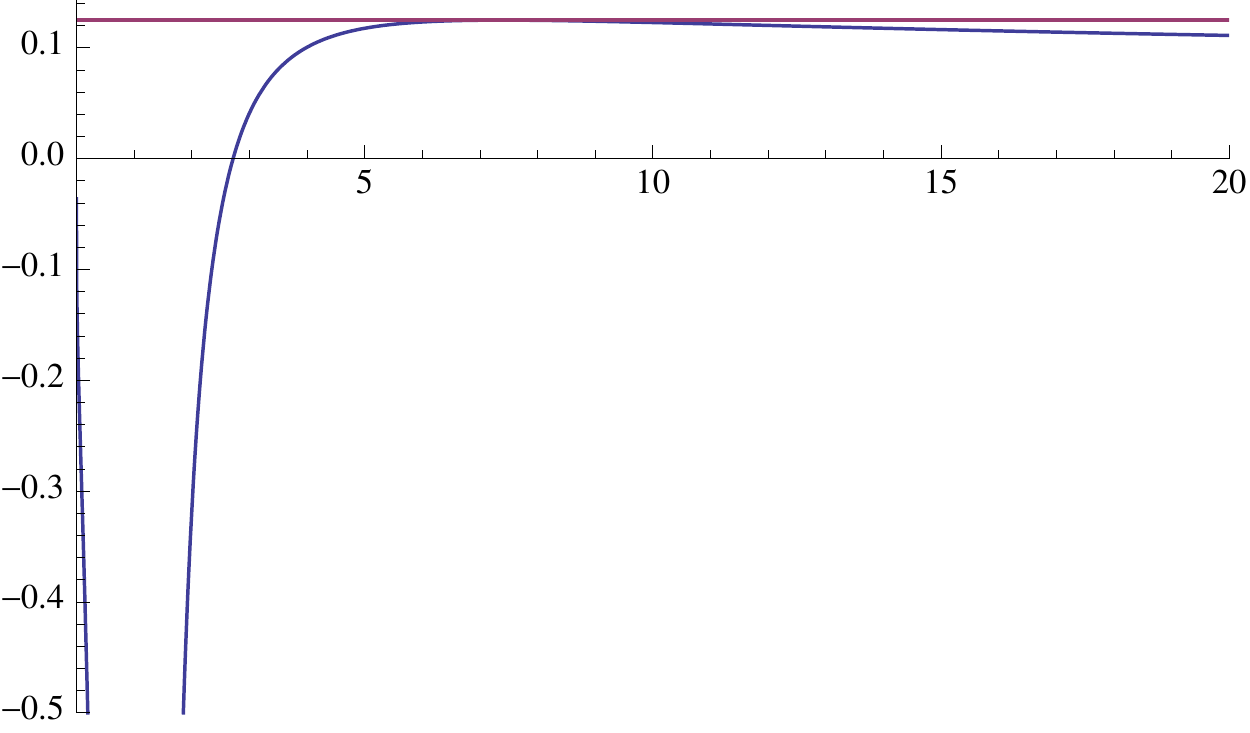} \hspace*{1cm}	
		\centering
		\caption{\footnotesize{The one dimensional representation of $-\frac{D_F^2(\norm{\log U}^2).(\xi\otimes\eta,\xi\otimes\eta)}{[D_F(\norm{\log U}^2).\xi\otimes\eta]^2}=-\frac{\frac{d^2}{dt^2}((\log t)^2)\,\xi^2}{[\frac{d}{dt}((\log t)^2)\,\xi]^2}=\frac{ \log t-1}{ 2\, (\log t)^2}$.} }\label{fig1d}
	\end{minipage}
\end{figure}
\begin{figure}[h!]
	\begin{minipage}[h]{1\linewidth}
		\centering
		\includegraphics[scale=0.7]{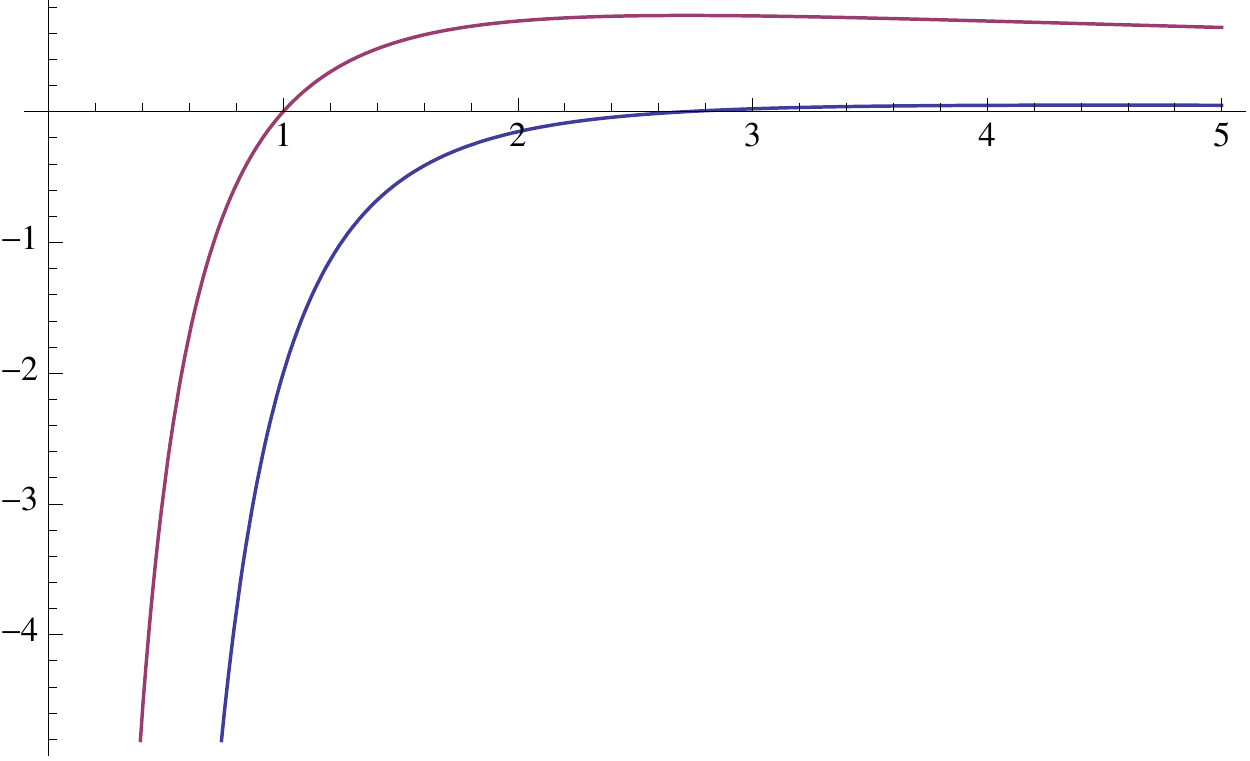} \hspace*{1cm}	
		\centering
		\caption{\footnotesize{The one dimensional representation of $t=F\mapsto-{D_F^2(\norm{\log U}^2).(\xi\otimes\eta,\xi\otimes\eta)}=-{\frac{d^2}{dt^2}((\log t)^2)}=-\frac{2(1-\log t)}{t^2}$ and of $t=F\mapsto{D_F(\norm{\log U}^2).\xi\otimes\eta}={\frac{d}{dt}((\log t)^2)}=\frac{2\log t}{t}$.} }\label{fig2d}
	\end{minipage}
\end{figure}

We observe first that $t\mapsto \frac{ \log t-1}{ 2\, (\log t)^2\,}$ is bounded above by $\frac{1}{8}$, see Fig.\ \ref{fig1d}, since
\begin{align}
\max_{t\in (0,\infty)}\frac{ \log t-1}{ 2\, (\log t)^2} = \frac{ \log t-1}{ 2\, (\log t)^2}\Big|_{t=e^2} = \frac{1}{8}\,.
\end{align}
In particular, there exists no \emph{concave critical point} of the mapping $t\mapsto (\log t)^2$, i.e.\ the conditions $-{\frac{d^2}{dt^2}((\log t)^2)}>0$ and ${\frac{d}{dt}((\log t)^2)}=\frac{2\log t}{t}=0$ are never satisfied for the same $t>0$.

If we also assume that $\Psi$ is monotone increasing and convex, then \eqref{eq:1dfullConvexityStatement} is always satisfied for $t<1$ and therefore reduces to the condition
\[
	\Psi^{\prime\prime}(s)\geq \frac{\sqrt{s}-1}{ 2\, s^2\,}\, \Psi^\prime(s) \quad\text{ for all }\; s>0
\]
which, for instance, is satisfied by any monotone convex function $\Psi\col\mathbb{R}\to \mathbb{R}$ with
\begin{align}
	\Psi^{\prime\prime}(s)\geq\frac{1}{8}\, \Psi^\prime(s) \qquad\text{for all }\; s>1\,.
\end{align}
For example, if $\Psi$ is given by $\Psi(s)=e^{\frac{1}{8}s}$, then the corresponding energy function $W\col\R\to\R$ with $W(x)=e^{\frac{1}{8}\.\log^2(x)}$ is convex with respect to $x$.

\subsection{Necessary conditions for rank-one convexity}\label{Logarsectmon}

Let $W\col\GLp(3)\to\mathbb{R}_+$ be an objective and isotropic function, and let $g$ denote its representation in terms of the singular values, i.e.\ $W(F)=g(\lambda_1,\lambda_2,\lambda_3)$ for all $F\in\GLp(3)$ with singular values $\lambda_1,\lambda_1,\lambda_3$. Then the \emph{Baker-Ericksen inequalities} can be stated as \cite{marsden1994foundations,bakerEri54}
\begin{align}
(\lambda_i-\lambda_j)\cdot \bigg( \lambda_i\frac{\partial g}{\partial \lambda_i}-\lambda_j\frac{\partial g}{\partial \lambda_j} \bigg) \geq 0 \qquad\text{for all }\;\lambda_i,\lambda_j\in(0,\infty)\,,\; i,j=1,2,3\,,
\end{align}
which is equivalent to \cite{silhavy2002monotonicity}
\begin{align}\label{inegg}
g(\lambda_1,\lambda_2,\lambda_3)\geq g(\overline{\lambda}_1,\overline{\lambda}_2,\overline{\lambda}_3)
\end{align}
for all $(\lambda_1,\lambda_2,\lambda_3)$ and $(\overline{\lambda}_1,\overline{\lambda}_2,\overline{\lambda}_3)$ such that
\[
\lambda_1\geq \lambda_2\geq \lambda_3, \qquad \overline{\lambda}_1\geq \overline{\lambda}_2\geq \overline{\lambda}_3, 
\]
and
\[
\lambda_1\geq \overline{\lambda}_1, \quad \lambda_1\lambda_2\geq \overline{\lambda}_1\overline{\lambda}_2, \quad \lambda_1\lambda_2\lambda_3\geq \overline{\lambda}_1\overline{\lambda}_2\overline{\lambda}_3.
\]
It is well known \cite{dacorogna01,silhavy1997mechanics} that the Baker-Ericksen inequalities are a necessary\footnote{%
	Note that the Baker-Ericksen inequalities are not sufficient for rank-one convexity; for example, the mapping $F\mapsto \norm{\log F^T F}^2$ is not rank-one convex, while the corresponding representation $g$ in terms of the singular values satisfies \eqref{inegg} \cite{agn_borisov2015sum}.%
}
condition for rank-one convexity of the energy $W$.
In particular, the rank-one convexity of $W$ therefore implies
\begin{align}
	g(\lambda_1,1,1)\geq g(\overline{\lambda}_1,1,1) \qquad\text{for all }\; \lambda_1,\overline{\lambda}_1\in \mathbb{R}_+\,,\; \lambda_1\geq\overline{\lambda}_1\geq 1\,,
\end{align}
i.e.\ that the mapping $x\mapsto g(x,1,1)$ is monotone increasing on $[1,\infty)$,
as well as 
\begin{align}
	g(1,1,\lambda_1)\geq g(1,1, \overline{\lambda}_1) \qquad\text{for all }\; \lambda_1,\overline{\lambda}_1\in \mathbb{R}_+\,,\; \lambda_1\leq\overline{\lambda}_1\leq 1\,,
\end{align}
i.e.\ that the mapping $x\mapsto g(1,1,x)$ is monotone decreasing on $(0,1]$. Since $g$ is symmetric, the mapping $x\mapsto g(x,1,1)$ is monotone decreasing on $(0,1]$ as well for rank-one convex energies.

If $W$ is of the form $W(F) = \Psi(\norm{\log U}^2) = g(\lambda_1,\lambda_2,\lambda_3)$ with a differentiable function $\Psi$, then the representation of $W(F)$ is terms of the singular values of $F$ is given by
\begin{align}
	g(\lambda_1,\lambda_2,\lambda_3)=\Psi(\log^2\lambda_1+\log^2\lambda_2+\log^2\lambda_3)\,.
\end{align}
Thus the rank-one convexity of such an energy implies that the mapping $x\mapsto g(x,1,1) = \Psi(\log^2 x)$ must be monotone increasing on $[1,\infty)$ and monotone decreasing on $(0,1]$. Since
\begin{align}
	\frac{d}{dx}\Psi(\log^2x)=\Psi^\prime(\log^2 x)\cdot \frac{2\.\log x}{x} %
\end{align}
it follows that the rank-one convexity of $F\mapsto W(F)=\Psi(\norm{\log U}^2)$ implies that $\Psi$ is monotone increasing on $[0,\infty)$. The same result holds true for arbitrary dimension $n$.

Similar conditions hold if $W$ is given in terms of the deviatoric logarithmic strain measure $\norm{\dev_3\log U}^2$, i.e.\ if $W$ is of the form $W(F)=\Psi(\norm{\dev_3\log U}^2)$. In that case,
\begin{align}
W(F)=g(\lambda_1,\lambda_2,\lambda_3)=\Psi \left( \frac{1}{3} \left( \log^2 \frac{\lambda_1}{\lambda_2}+\log^2 \frac{\lambda_2}{\lambda_3}+\log^2 \frac{\lambda_3}{\lambda_1} \right) \right)\,,
\end{align}
thus rank-one convexity of $W$ implies that the mapping $x\mapsto g(x,1,1)=\Psi(\frac{2}{3}\log^2 x)$ is monotone increasing on $[1,\infty)$ and monotone decreasing on $(0,1]$, which, in turn, implies that $\Psi$ must be monotone increasing on $[0,\infty)$. Again, the same monotonicity condition must be satisfied for arbitrary dimension $n$.

A similar implication was found by Sendova and Walton \cite{sendova2005strong}, who considered energy functions of the form
\[
	W(F) = \Phi \left( \log \det U,\; \norm{\dev_3 \log U},\; 3\.\sqrt{6}\.\det \left(\frac{\dev_3 \log U}{\norm{\dev_3\log U}} \right) \right)\,.
\]
In particular \cite[Proposition 2]{sendova2005strong}, they showed that if a mapping of the form $F\mapsto\widetilde{\Psi}(\norm{\dev_3 \log U})$ is Legendre-Hadamard elliptic (i.e.\ rank-one convex), then\footnote{%
Although Sendova and Walton considered strong Legendre-Hadamard ellipticity and deduced a strict version of the inequalities \eqref{noi}, their proof can easily be seen to work for (non-strict) rank-one convexity as well.
}%
\begin{align}\label{noi}
	\widetilde{\Psi}^\prime(t) \geq 0\qquad \text{and}\qquad \widetilde{\Psi}^{\prime\prime}(t)\geq \left(\frac{3\,t}{8}+\frac{1}{t}\right)\widetilde{\Psi}^\prime(t)
\end{align}
for all $t>0$. Of course, for $W(F)=\Psi(\norm{\dev_3\log U}^2)=\widetilde{\Psi}(\norm{\dev_3\log U})$, the representations $\Psi$ and $\widetilde{\Psi}$ are connected by the equality
\[
	\Psi(t^2) = \widetilde{\Psi}(t)\,,
\]
thus rank-one convexity of $W$ implies
\[
	0 \leq \widetilde{\Psi}'(t) = 2\.t\,\Psi'(t^2) \quad\implies\quad 0 \leq \Psi'(t^2)
\]
as well as
\[
	4\.t^2\,\Psi''(t^2) + 2\.\Psi'(t^2) =\widetilde{\Psi}''(t) \geq \left(\frac{3\.t}{8}+\frac{1}{t}\right)\widetilde{\Psi}'(t) =
	\left(\frac{3\.t}{8}+\frac{1}{t}\right)\cdot 2\.t\.\Psi'(t^2) =
	\left(\frac{3\.t^2}{4}+2\right)\Psi'(t^2)
\]
or, equivalently,
\begin{align}
	\Psi''(t^2) \;\geq\; \frac{3}{16}\.\Psi'(t^2)
\end{align}
for all $t>0$. In particular, rank-one convexity of $F\mapsto \Psi(\norm{\dev_3\log U}^2)$ therefore implies
\[
	\Psi'(x) \geq 0 \qquad\text{and}\qquad \Psi''(x) \geq 0 \qquad\text{ for all }\;x>0\,.
\]

Moreover, it can be inferred from \cite[Proposition 2]{sendova2005strong} that a necessary condition for \emph{strict} Legendre-Hadamard ellipticity of an energy $W$ with $W(F)=\Psi(\norm{\dev_3\log U}^2)$ is that $\Psi\col[0,\infty)\to \mathbb{R}$ must be strictly monotone and uniformly convex. We note that the strict monotonicity of $\Psi$ also follows from the (not necessarily strict) Legendre-Hadamard ellipticity (i.e.\ from classical rank-one convexity) if $\Psi$ is two-times continuously differentiable and $\Psi'(0)>0$, since, in that case, the convexity of $\Psi$ implies $\Psi'(x)\geq\Psi'(0)>0$ for all $x>0$.

The results of this section are summarized in the following lemmas.

\begin{lemma}
\label{lemma:monotonicity}
	Let $\Psi\col[0,\infty)\to\R$ be continuously differentiable such that the mapping $F\mapsto \Psi(\norm{\log U}^2)$ is rank-one convex on $\GLpn$. Then $\Psi$ is monotone increasing.
\end{lemma}
\begin{lemma}
\label{lemma:monotonicityDev}
	Let $\Psi\col[0,\infty)\to\R$ be continuously differentiable such that the mapping $F\mapsto \Psi(\norm{\dev_n \log U}^2)$ is rank-one convex on $\GLpn$. Then $\Psi$ is monotone increasing.
\end{lemma}
\begin{lemma}
\label{lemma:strictMonotonicityConditionDev}
	Let $\Psi\col[0,\infty)\to\R$ be two-times continuously differentiable such that the mapping $F\mapsto \Psi(\norm{\dev_3 \log U}^2)$ is rank-one convex on $\GLp(3)$. Then
	\begin{itemize}
		\item[i)] $\Psi$ is convex,
		\item[ii)] if $\Psi'(0)>0$, then $\Psi'(x)>0$ for all $x>0$.
	\end{itemize}
\end{lemma}
\begin{remark}
\label{remark:strictMonotonicityConditionLinearCompatibilityDev}
	The requirement $\Psi'(0)>0$ in Lemma \ref{lemma:strictMonotonicityConditionDev} is necessarily satisfied if the elastic energy $\Wiso\col\GLp(3)\to\R$ with $W(F)=\Psi(\norm{\dev_3\log U}^2)$ is of the form
	\begin{equation}
		\label{eq:linearCompatibilityDev}
		\Wiso(\id+H) = \mu\.\norm{\dev_3 \sym H}^2 + \mathcal{O}(\norm{H}^3)
	\end{equation}
	with $\mu>0$, where $\sym H = \frac12(H+H^T)$ is the symmetric part of $H\in\R^{3\times3}$, since
	\begin{align*}
		D_F^2 W(\id).(H,H)
		&=\Psi''(0)\cdot
			\smash{%
				\underbrace{%
					[(D_F \norm{\dev_3\log U}^2|_{F=\id}).H]^2
				}_{=0 \vphantom{=2\.\norm{\dev_3 \sym H}^2}}
				\,+\, \Psi'(0)\cdot
				\underbrace{%
					(D_F^2 \norm{\dev_3\log U}^2|_{F=\id}).(H,H)
				}_{=2\.\norm{\dev_3 \sym H}^2}
			}
	\end{align*}
	for $H\in\R^{3\times3}$.
	
	Since a function $\Wiso$ depending only on $\dev_3 \log U$ is always \emph{isochoric}, i.e.\ $\Wiso(a\.F)=\Wiso(F)$ for all $a>0$, it is often coupled additively with a \emph{volumetric} function depending on $\det F = \det U$ to obtain a viable elastic energy potential $W$ of the form $W(F)=\Wiso(F)+\Wvol(\det F)$. In that case, $W$ can only be \emph{compatible with classical linear elasticity}, i.e.\ be of the form
	\begin{equation}
	\label{eq:linearCompatibility}
		W(\id+H) = \mu\.\norm{\dev_3 \sym H}^2 + \frac{\kappa}{2}\.[\tr(\sym H)]^2 + \mathcal{O}(\norm{H}^3)
	\end{equation}
	with $\mu>0$ and $\kappa>0$, if \eqref{eq:linearCompatibilityDev} is satisfied for $\Wiso$ and thus if $\Psi'(0)>0$. This so-called \emph{volumetric-isochoric split} is discussed further in Section \ref{volisosect}.
\end{remark}

\section{Functions depending on $\norm{\log U}^2$}\label{log}

Although it was shown in the previous section that if an energy $W$ of the form $W(F) = \Psi(\norm{\log U}^2)$ is to be rank-one convex on $\GLpn$ the function $\Psi$ must be monotone increasing, we will assume in the following that the monotonicity of $\Psi$ is strict. In particular, this restriction excludes the trivial examples of constant functions $\Psi$, for which the energy $W$ would obviously be rank-one convex and polyconvex. Our main result is the following.
\begin{proposition}\label{thlog}
	There is no strictly monotone function $\Psi\col[0,\infty)\to\mathbb{R}$ such that
	\begin{align}\label{fde1}
	F\mapsto W(F) = \Psi(\norm{\log U}^2) = \Psi(\norm{\log V}^2)
	\end{align}
	is rank-one convex in $\GLp(n)$, $n\geq 2$.
\end{proposition}
\begin{proof}
Our aim is to use Lemma \ref{lemma:impossibility} and, therefore, to show that there exist $F\in\GLpn$ and $\xi,\eta\in\mathbb{R}^n\setminus \{0\}$ such that 
\begin{align}\label{ident1}
	D_F(\norm{\log U}^2).(\xi\otimes\eta)=0 \qquad \text{and}\qquad D_F^2(\norm{\log U}^2).(\xi\otimes\eta,\xi\otimes\eta)<0\,.
\end{align}
Since (see Appendix \ref{appendix:logDerivative})
\[
	D_F(\norm{\log U}^2).(\xi\otimes \eta) = \iprod{2\,(\log V)\, F^{-T}, \xi\otimes \eta}
\]
for all $F\in\GLpn$ and all $\xi,\eta\in\R^n$, the conditions \eqref{ident1} are satisfied if there are $F\in \GLp(n)$ and two directions $\xi,\eta\neq 0$ such that
\begin{align}
\langle (\log V)\, \xi, F^{-T}\, \eta\rangle=0,
\qquad 
\text{and}
\qquad 
D_F^2(\norm{\log V}^2).(\xi\otimes\eta,\xi\otimes\eta)<0.
\end{align}

It is difficult to compute $D_F^2(\norm{\log V}^2).(\xi\otimes\eta,\xi\otimes\eta)$ explicitly without resorting to a cumbersome eigenvector representation.
However, using the function $h\col(-\varepsilon,\varepsilon)\to \mathbb{R}_+$ with
\begin{align}
h(t)&=\norm{\log \sqrt{(F+t\,\xi\otimes\eta)^T(F+t\,\xi\otimes\eta)}}^2=\norm{\log \sqrt{(F+t\,\xi\otimes\eta)(F+t\,\xi\otimes\eta)^T}}^2\notag\\
&=
\frac{1}{4}\norm{\log (F+t\,\xi\otimes\eta)(F+t\,\xi\otimes\eta)^T}^2=\frac{1}{4}\sum\limits_{i=1}^n\log^2 \mu_i(t),
\end{align}
where
$\mu_i(t)$, $i=1,2,...,n$ are the eigenvalues of $(F+t\,\xi\otimes\eta)(F+t\,\xi\otimes\eta)^T$ and $\varepsilon>0$ is small enough such that $\mu_i(t)>0$ for all $i=1,2,...,n$ and all $t\in(-\varepsilon,\varepsilon)$, we may write
\begin{align}
h'(t)&=
\frac{1}{4}\langle D_F (\norm{\log (F+t\,\xi\otimes\eta)(F+t\,\xi\otimes\eta)^T}^2),\xi \otimes \eta\rangle,\\
h''(t)&=
\frac{1}{4} D^2_F (\norm{\log (F+t\,\xi\otimes\eta)(F+t\,\xi\otimes\eta)^T}^2). (\xi \otimes \eta,\xi \otimes \eta).\notag
\end{align}
Hence, since
\begin{align}
h'(0)&=
\frac{1}{4}\langle D_F (\norm{\log { F\,F^T}}^2),\xi \otimes \eta\rangle \qquad\text{and}\qquad
h''(0)=
\frac{1}{4} D^2_F (\norm{\log { F\, F^T}}^2). (\xi \otimes \eta,\xi \otimes \eta)\,,
\end{align}
the conditions \eqref{ident1} are satisfied if
\begin{align}\label{condh}
 h''(0)< 0\qquad \text{and}\qquad h'(0)&=
 \frac{1}{4}\langle (\log V)\, \xi, F^{-T}\, \eta\rangle=0\,.
\end{align}

We note that once the result is established for $n=2$, it can immediately be extended to arbitrary dimension $n$ (by suitable restriction).
\begin{figure}[h!]
	\begin{minipage}[h]{1\linewidth}
		\centering
		\vspace*{1cm}
		\includegraphics[scale=0.6,angle =0]{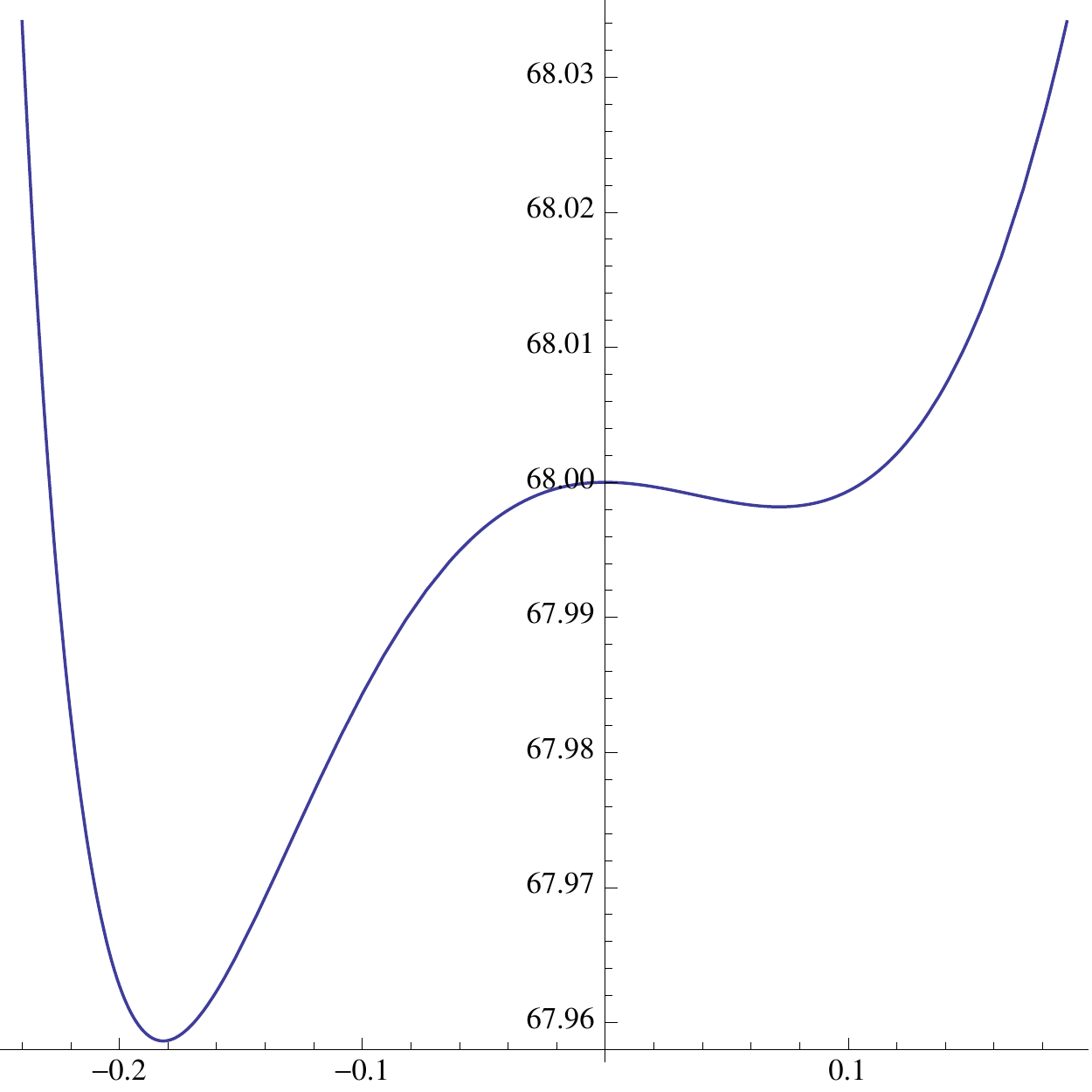}	
		\centering
		\caption{\footnotesize{The function $h$ has a critical point and is concave at $t=0$.} }\label{region}
	\end{minipage}
\end{figure}
In the two-dimensional case, the equation 
$\langle (\log V)\, \xi, F^{-T}\, \eta\rangle=0$
is satisfied if
\begin{align}\label{zo}
 F^{-T}\, \eta&=a\,\begin{pmatrix}
 ((\log V)\, \xi)_2\\-((\log V)\, \xi)_1
 \end{pmatrix} \quad\iff\quad \eta=a\, F^{T}\,\begin{pmatrix}
 ((\log V)\, \xi)_2\\-((\log V)\, \xi)_1
 \end{pmatrix}=a\, F^{T}\,\begin{pmatrix}
 0&1\\
 -1&0
 \end{pmatrix}(\log V)\, \xi
\end{align}
for some number $a\in\R$.

Let $F=\matr{e^8&0\\0&e^2}$ and $\xi=\frac{\sqrt{2}}{2}\matr{1\\1}$. Then for $a=-1$, \eqref{zo} yields $\eta=\sqrt{2}\.\matr{-e^8\\4\.e^2}$. The eigenvalues $\mu_i(t)$ of $(F+t\,\xi\otimes\eta)(F+t\,\xi\otimes\eta)^T$ are given by
\begin{align}
	\mu_1(t) &= \frac{e^4}{2}\, \left( 1 + 32\.t^2 + 8\.t + e^{12}\.(2\.t^2 - 2\.t + 1) + \sqrt{(32\.t^2 + 8\.t + 1 + e^{12}\.(2\.t^2 - 2\.t +1))^2 - 4\.e^{12}\.(3\.t + 1)^2} \,\right)\,,\nnl
	\mu_2(t) &= \frac{e^4}{2}\, \left( 1 + 32\.t^2 + 8\.t + e^{12}\.(2\.t^2 - 2\.t + 1) - \sqrt{(32\.t^2 + 8\.t + 1 + e^{12}\.(2\.t^2 - 2\.t +1))^2 - 4\.e^{12}\.(3\.t + 1)^2} \,\right)\,. \nonumber
\end{align}
Since $\mu_1(t)\,\mu_2(t)=[\det (F + t\,\xi\otimes\eta)]^2 = e^{20}\.(3 t+1)^2$ and $\mu_1(t)>0$ for all $t\in \mathbb{R}$, the function $h$ has the form
\begin{equation*}
	h(t) = \frac{1}{4}\, \left[ \log^2 \mu_1(t)+\log^2 \left( e^{20}\.\frac{(3t+1)^2}{\mu_1(t)} \right) \right] = \frac{1}{4}\, \left[ \log^2 \mu_1(t) + \Big( 20 + 2\.\log(3t+1) - \log\mu_1(t) \Big)^2 \right]
\end{equation*}
for $t\in(-\frac13,\frac13)$, and thus its derivatives are given by
\begin{align}
	h'(t) &= \frac14\, \left[ \frac{2\.\mu_1'(t)\cdot \log \mu_1(t)}{\mu_1(t)} + \Big( 40 + 4\.\log(3\.t+1)-2\.\log(\mu_1(t)) \Big) \cdot \left( \frac{6}{3\.t+1}-\frac{\mu_1'(t)}{\mu_1(t)} \right) \right]\,,\\
	h''(t) &= \frac14\, \Bigg[ \frac{2\.\mu_1''(t)\,\log\mu_1(t)}{\mu_1(t)} + \frac{2\.\mu_1'(t)^2}{(\mu_1(t))^2} + 2\,\left( \frac{6}{3\.t+1} - \frac{\mu_1'(t)}{\mu_1(t)} \right)^2 - \frac{2\.\mu_1'(t)^2\,\log\mu_1(t)}{(\mu_1(t))^2}\nnl
		&\qquad\quad + \Big( 40 + 4\.\log(3\.t+1)-2\.\log(\mu_1(t)) \Big) \cdot \left( \frac{\mu_1'(t)^2}{(\mu_1(t))^2} - \frac{\mu_1''(t)}{\mu_1(t)} - \frac{18}{(3\.t+1)^2} \right) \Bigg]\,.
\end{align}
In particular, since
\begin{align}
	\mu_1(0)=e^{16}\,, \qquad \mu_1'(0)=-2\, e^{16}\,, \qquad \mu_1''(0)=\frac{2\, e^{16} \left(7+2\, e^{12}\right)}{e^{12}-1}\,,
\end{align}
we find (cf.\ Fig.\ \ref{region})
\begin{align}
	h'(0)&=0,\qquad 
	h''(0)=\frac{110-2\.e^{12}}{e^{12}-1}<0\,.
\end{align}
In conclusion, for $\xi=\frac{\sqrt{2}}{2}\matr{1\\1}$, $\eta= \sqrt{2}\.\matr{4\,e^2 \\
- e^8}$ and $F=\matr{e^8&0\\0&e^2}$, the desired conditions \eqref{ident1} are satisfied. Hence, according to Lemma \ref{lemma:impossibility}, the function $W$ cannot be rank-one convex.
\end{proof}

\section{Functions depending on $\norm{\dev_n\log U}^2$}\label{devlog}

We now consider a function $W\col\GLpn\to\R$ of the form
\begin{align}\label{devf}
	W(F)&=\Psi(\norm{\dev_n\log U}^2)=\Psi(\norm{\dev_n\log V}^2)=\Psi\left(\frac{1}{n}\sum\limits_{i,j=1}^n\log^2 \frac{\mu_i}{\mu_j}\right)\,,
\end{align}
where $\mu_i$, $i=1,2,...,n$ are the singular values of $F$.

Since $\dev_n \log (U\inv) = -\dev_n \log U$ and $\dev_n (a\.\log U) = \dev_n \log U$ for $a>0$, it is easy to see that every function $W$ of the form \eqref{devf} is tension-compression symmetric and isochoric, i.e.\ satisfies
\begin{align}
	W(F)=W(F^{-1})\qquad \text{and}\qquad W(a\.F)=W(F)
\end{align}
for all $F\in\GLpn$ and all $a>0$; note that, in particular,
\begin{align}
	W(F) = W\left(\frac{F}{(\det F)^{\afrac1n}}\right) \qquad\text{ for all }\; F\in\GLp(n)\,.
\end{align}
Furthermore, in the planar case, i.e.\ for $n=2$, every objective, isotropic and isochoric energy $W\col\GLpz\to\R$ can be written in the form \eqref{devf} with a unique function $\Psi\col[0,\infty)\to\R$, and the rank-one convexity is characterized by the following result \cite{agn_martin2015rank,agn_ghiba2017SL}:
\begin{proposition}
	\label{prop:mainResultInTermsOfLogSquared}
	Let $W\col\GLpz\to\R,\;F\mapsto W(F)$ be an objective, isotropic and isochoric function and let $\Psi\col[0,\infty)\to\R$ denote the uniquely determined functions with
	\[
	W(F)=\Psi(\norm{\dev_2\log U}^2)
	\]
	for all $F\in\GLpz$ with singular values $\lambda_1,\lambda_2$. If $\Psi\in C^2([0,\infty))$, then the following are equivalent:
	\begin{itemize}
		\item[i)] $W$ is polyconvex,
		\item[ii)] $W$ is rank-one convex,
		\item[iii)] $2\,\eta\,\Psi^{\prime\prime}(\eta)+ (1-\sqrt{2\,\eta})\,\Psi^{\prime}(\eta)\geq 0$ \quad for all $\eta\in(0,\infty)$.
	\end{itemize}
\end{proposition}
\noindent
For example, the energy $W\col\GLpz\to\R$ with $W(F)=e^{k\norm{\dev_2\log U}^2}$ is polyconvex for $k\geq \frac{1}{4}$.

In the three-dimensional case, however, not every function $W$ of the form $W(F)=\Psi(\norm{\dev_3\log V}^2)$ such that $\Psi$ satisfies condition iii) in Proposition \ref{prop:mainResultInTermsOfLogSquared} is polyconvex or even rank-one convex (e.g.\ the mapping $F\mapsto e^{k\norm{\dev_3\log V}^2}$). In fact, there exists no strictly monotone function $\Psi$ such that $W$ is rank-one convex, as the following result shows.

\begin{proposition}\label{notrankonedev}
	For $n\geq3$, there is no strictly monotone function $\Psi\col[0,\infty)\to\mathbb{R}$ such that
	\begin{align}\label{fde2}
	F\mapsto W(F)=\Psi(\norm{\dev_n\log V}^2)
	\end{align}
	is rank-one convex in $\GLp(n)$.
\end{proposition}
\begin{remark}
	According to Remark \ref{remark:strictMonotonicityConditionLinearCompatibilityDev}, for a sufficiently smooth function $W$ on $\GLp(3)$, the condition of strict monotonicity can be replaced by the requirement that $W$ is compatible with linear elasticity.
\end{remark}

\begin{proof}
Without loss of generality, we consider only the case $n=3$, since the result may be extended to arbitrary dimension $n\geq3$ by a suitable restriction. The idea of the proof is similar to that of Proposition \ref{thlog}, i.e.\ we need to find $F\in\GLpn$ and $\xi,\eta\in\mathbb{R}^n\setminus \{0\}$ such that 
\begin{align}\label{ident2}
	D_F(\norm{\dev_3\log U}^2).(\xi\otimes\eta)=0 \qquad \text{and}\qquad D_F^2(\norm{\dev_3\log U}^2).(\xi\otimes\eta,\xi\otimes\eta)<0\,.
\end{align}
Since
\begin{align}
\label{eq:firstDerivativeDevLog}
	D_F(\norm{\dev_3\log U}^2).(\xi\otimes\eta) = \langle 2\,(\dev_3\log V)\, F^{-T}, \xi\otimes \eta\rangle = 2\, \langle (\dev_3\log V)\.\xi,\, F^{-T}\eta \rangle\,,
\end{align}
conditions \eqref{ident2} can be restated as
\begin{align}\label{dasp}
	\langle (\dev_3\log V)\, \xi, F^{-T}\, \eta\rangle=0
	\qquad 
	\text{and}
	\qquad 
	D_F^2(\norm{\dev_3\log V}^2).(\xi\otimes\eta,\xi\otimes\eta)<0.
\end{align}
 
 For given fixed $F\in \GLp(3)$ and $\xi\in \mathbb{R}^3$ such that $(\dev_3\log V)\,\xi\neq 0$, a solution $\eta\in \mathbb{R}^3$ of equation \eqref{dasp}$_1$ is given by
\begin{align}\label{m0}
F^{-T}\eta=\begin{pmatrix}
0&1&0\\
-1&0&0\\
0&0&0
\end{pmatrix}\,\vartheta \;\equalscolon\; m_0, \qquad \vartheta=\frac{(\dev_3\log V)\, \xi}{\norm{(\dev_3\log V)\,\xi}} \,.
\end{align}
\begin{figure}[b!]
	\begin{minipage}[h]{1\linewidth}
		\centering
		\includegraphics[scale=0.20,angle =0]{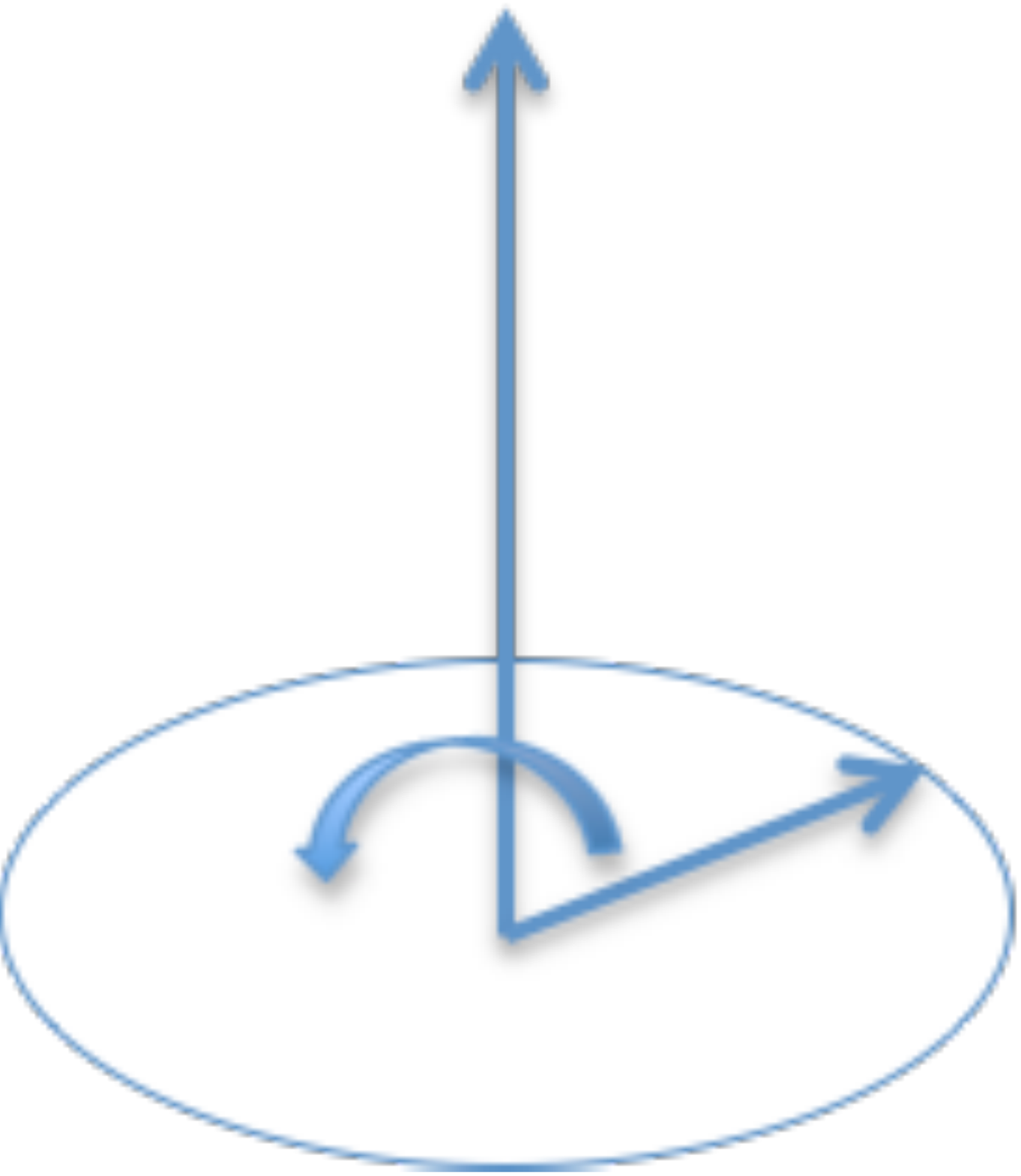}
		\put(-49,119){$\vartheta=\frac{(\dev_3\log V)\, \xi}{\norm{(\dev_3\log V)\,\xi}}$}
		\put(-70,50){$\theta$}	
		\put(-10,60){$m_0=F^{-T}\eta$}	
		\centering
		\caption{\footnotesize{Construction of the counterexample.} }\label{rotire}
	\end{minipage}
\end{figure}

\noindent
More generally, any $m\in \mathbb{R}$ obtained by arbitrary rotation $Q(\vartheta,\theta)$, $\theta\in [0,2\, \pi)$ of $m_0$ (as given in \eqref{m0}) around $\vartheta$ provides a solution of \eqref{dasp}$_1$, i.e.\ for given fixed $F\in \GLp(3)$ and $\xi\in \mathbb{R}^3$, any $\eta\in \mathbb{R}^3$ given by
\begin{align}
F^{-T}\eta&=m=Q(\vartheta,\theta)\,m_0= Q(\vartheta,\theta)\,\begin{pmatrix}
0&1&0\\
-1&0&0\\
0&0&0
\end{pmatrix}\vartheta\,, \notag\\
\implies\quad &\eta=F^{T} Q(\vartheta,\theta)\,\begin{pmatrix}
0&1&0\\
-1&0&0\\
0&0&0
\end{pmatrix}\,\frac{(\dev_3\log V)\, \xi}{\norm{(\dev_3\log V)\,\xi}} \label{eq:orthogonalEta}
\end{align}
is a solution to equation \eqref{dasp}$_1$. Recall that for a unit vector $\vartheta=(\vartheta_1,\vartheta_2,\vartheta_3)^T$, the matrix for a rotation by an angle $\theta\in[0,2\,\pi)$ about the axis $\vartheta$ is given by
\begin{align} Q(\vartheta,\theta)&=\begin{pmatrix} 
	\cos \theta +\vartheta_1^2 \left(1-\cos \theta\right) & \vartheta_1 \vartheta_2 \left(1-\cos \theta\right) - \vartheta_3 \sin \theta & \vartheta_1 \vartheta_3 \left(1-\cos \theta\right) + \vartheta_2 \sin \theta \\
	\vartheta_2 \vartheta_1 \left(1-\cos \theta\right) + \vartheta_3 \sin \theta & \cos \theta + \vartheta_2^2\left(1-\cos \theta\right) & \vartheta_2 \vartheta_3 \left(1-\cos \theta\right) - \vartheta_1 \sin \theta \\ 
	\vartheta_3 \vartheta_1 \left(1-\cos \theta\right) - \vartheta_2 \sin \theta & \vartheta_3 \vartheta_2 \left(1-\cos \theta\right) + \vartheta_1 \sin \theta & \cos \theta + \vartheta_3^2\left(1-\cos \theta\right)
	\end{pmatrix} \notag\\
	&= \cos\theta\, \id_3 + \sin\theta\,{\rm anti}(\vartheta) + (1-\cos\theta)\,{\vartheta}\otimes{\vartheta}\,, \label{eq:rotationMatrix}
\end{align}
where
\begin{align} {\rm anti}(\vartheta)= \begin{pmatrix}
	0 & -\vartheta_3 & \vartheta_2 \\[3pt]
	\vartheta_3 & 0 & -\vartheta_1 \\[3pt]
	-\vartheta_2 & \vartheta_1 & 0
	\end{pmatrix}
\end{align}
is the cross-product matrix of $\vartheta$.

Again, computing $D_F^2(\norm{\dev_3\log V}^2).(\xi\otimes\eta,\xi\otimes\eta)$ explicitly is rather inconvenient. We therefore introduce the function $h\col(-\varepsilon,\varepsilon)\to \mathbb{R}_+$ given by
\begin{align}
	h(t)&=\norm{\dev_n\log \sqrt{(F+t\,\xi\otimes\eta)(F+t\,\xi\otimes\eta)^T}}^2=
	\frac{1}{4}\norm{\dev_3\log (F+t\,\xi\otimes\eta)(F+t\,\xi\otimes\eta)^T}^2\notag\\
	&=\frac{1}{12}\left(\log^2 \frac{\mu_1}{\mu_2}+\log^2 \frac{\mu_2}{\mu_3}+\log^2 \frac{\mu_3}{\mu_1}\right)\,, \label{eq:devLogRatios}
\end{align}
where
$\mu_i(t)$ , $i=1,2,3$ are the eigenvalues of $(F+t\,\xi\otimes\eta)(F+t\,\xi\otimes\eta)^T$ and $\varepsilon>0$ is small enough such that $\mu_i(t)>0$, $i=1,2,...,n$ for all $t\in(-\varepsilon,\varepsilon)$. Then
\begin{align}
	h'(t)&=
	\frac{1}{4}\langle D_F (\norm{\dev_n\log (F+t\,\xi\otimes\eta)(F+t\,\xi\otimes\eta)^T}^2),\xi \otimes \eta\rangle\,,\notag\\
	h''(t)&=
	\frac{1}{4} D^2_F (\norm{\dev_n\log (F+t\,\xi\otimes\eta)(F+t\,\xi\otimes\eta)^T}^2). (\xi \otimes \eta,\xi \otimes \eta) \label{eq:hDevLogSecondDerivative}
\end{align}
and thus
\begin{align}
	h'(0)&=
	\frac{1}{4}\langle D_F (\norm{\dev_n\log { F\,F^T}}^2),\xi \otimes \eta\rangle\,,\qquad 
	h''(0)=
	\frac{1}{4} D^2_F (\norm{\dev_n\log { F\,F^T}}^2). (\xi \otimes \eta,\xi \otimes \eta)\,. \label{eq:hDevLogSecondDerivativeZero}
\end{align}
Due to \eqref{eq:orthogonalEta} and \eqref{eq:hDevLogSecondDerivativeZero}, %
in order to fulfil \eqref{dasp}, it is sufficient to find $F\in\GLp(3)$, $\xi\in\R^3$ and $\theta\in[0,2\pi)$ such that %
\begin{align}
	\label{ineq223}
	h''(0) < 0
	\qquad \text{for}\qquad 
	\eta&=F^{T} Q(\vartheta,\theta)\,\begin{pmatrix}
	0&1&0\\
	-1&0&0\\
	0&0&0
	\end{pmatrix}\frac{(\dev_3\log V)\, \xi}{\norm{(\dev_3\log V)\,\xi}}\,,
\end{align}
Let
\begin{align}
	\xi=\left(
	\begin{array}{c}
	0 \\
	\frac{1}{\sqrt{2}} \\
	\frac{1}{\sqrt{2}} \\
	\end{array}
	\right)\,,
	\qquad 
	F=\begin{pmatrix}
	1&0&0\\
	0&e^{20}&0\\
	0&0& e^{15}
	\end{pmatrix}\quad \text{and}\quad 
	\qquad
	\theta=\frac{\pi}{2}\,.
\end{align}
Then 
\begin{equation*}
	\vartheta \;=\; \frac{(\dev_3\log V)\, \xi}{\norm{(\dev_3\log V)\,\xi}} \;=\; \frac{1}{\sqrt{29}}\,\left(\begin{array}{c}0\\5\\2\end{array}\right)\,,
	\qquad
	Q(\vartheta,\theta) \;=\; {\rm anti}(\vartheta)+\vartheta\otimes \vartheta \;=\; \frac{1}{\sqrt{29}}\,\left(\begin{array}{ccc}0 & -2 & 5 \\ 2 & 25 & 10 \\ -5 & 10 & 4\end{array}\right)
\end{equation*}
and
\begin{equation}
	\eta \;=\; F^{T} Q(\vartheta,\theta)\,\begin{pmatrix}0&1&0\\-1&0&0\\0&0&0\end{pmatrix}\,\frac{(\dev_3\log V)\, \xi}{\norm{(\dev_3\log V)\,\xi}} \;=\; \frac{1}{29}\left(
	\begin{array}{c}0\\10\, e^{20}\\-25\, e^{15}\end{array}\right)\,.
\end{equation}
For these choices and for $t\in \left(-\frac{29 \sqrt{2}}{15},\frac{29 \sqrt{2}}{15}\right)$, we directly compute
\begin{align}
	\mu_1(t)&=1,\notag\\
	\mu_2(t)&=\frac{1}{1682}\Bigg[e^{40} \left(100\, t^2+290\, \sqrt{2}\, t+841\right)+e^{30} \left(625\, t^2-725\, \sqrt{2}\, t+841\right)\\& -\frac{1}{2} \sqrt{4\, e^{60}\, \left(e^{10} \left(100\, t^2+290\, \sqrt{2} t+841\right)+625\, t^2-725 \,\sqrt{2} t+841\right)^2-6728\, e^{70} \left(15 \,t-29\, \sqrt{2}\right)^2}\Bigg],\notag\\
	\mu_3(t)&=\frac{1}{1682}\Bigg[e^{40} \left(100\, t^2+290\, \sqrt{2}\, t+841\right)+e^{30} \left(625\, t^2-725\, \sqrt{2}\, t+841\right)\notag\\& +\frac{1}{2} \sqrt{4\, e^{60}\, \left(e^{10} \left(100\, t^2+290\, \sqrt{2} t+841\right)+625\, t^2-725 \,\sqrt{2} t+841\right)^2-6728\, e^{70} \left(15 \,t-29\, \sqrt{2}\right)^2}\Bigg]\,.\notag
\end{align}
Since $\mu_2(t)\,\mu_3(t)=[\det (F + t\,\xi\otimes\eta)]^2=\frac{e^{70} \left(15 \,t-29 \sqrt{2}\right)^2}{1682}$, and due to \eqref{eq:devLogRatios}, the function $h$ is given by
\begin{align}
	h(t)=\frac{1}{12} \left[\log^2 (\mu_3(t))+\log^2\left(\mu_3(t)\,\frac{1682}{e^{70} \left(15\, t-29 \sqrt{2}\right)^2}\right)+ \log^2\left(\mu_3^2(t)\,\frac{1682}{e^{70} \left(15\, t-29 \sqrt{2}\right)^2}\right)\right]\,,
\end{align}
and thus
\begin{align}
	h'(t)&=\frac{1}{6 \left(29 \,\sqrt{2}-15\, t\right) \mu_3(t)}\Bigg[3 \left(29 \,\sqrt{2}-15 \,t\right) \mu_3'(t) \log \frac{1682\, \mu_3^2(t)}{e^{70} \left(29\, \sqrt{2}-15\, t\right)^2}\notag\\&\qquad\qquad\qquad\qquad\qquad\qquad+30\, \mu_3(t) \log \frac{1682^2 \mu_3^3(t)}{e^{140} \left(29 \,\sqrt{2}-15\, t\right)^4}\Bigg]\,,\notag\\
	h''(t)&=\frac{1}{2 \left(29\sqrt{2}-15 \,t\right)^2 \mu_3^2(t)}\Bigg[-\left(29 \sqrt{2}-15\, t\right)^2 (\mu_3'(t))^2 \left(\log \frac{1682 \,\mu_3^2(t)}{e^{70} \left(29 \sqrt{2}-15\, t\right)^2}-2\right)\\&\qquad\qquad\qquad\qquad\qquad\qquad + \left(29 \sqrt{2}-15\, t\right)^2 \mu_3(t)\,\mu_3''(t) \log \frac{1682\, \mu_3^2(t)}{e^{70} \left(29 \sqrt{2}-15\, t\right)^2}\notag\\&\qquad\qquad\qquad\qquad\qquad\qquad+60 \left(29 \sqrt{2}-15\, t\right) \mu_3(t)\,\mu_3'(t)+150\,\mu_3^2(t) \left(\log \frac{1682^2 \mu_3^3(t)}{e^{140} \left(29 \sqrt{2}-15\, t\right)^4}+4\right)\Bigg]\,.\notag
\end{align}
At $t=0$, we find
\begin{align}
h'(0)&=\frac{1}{174 \sqrt{2}\, \mu_3(0)}\Big[87 \,\sqrt{2} \,\mu_3'(0) \log \frac{\mu_3^2(0)}{e^{70}}+30\, \mu_3(0) \,\log \frac{\mu_3^3(0)}{e^{140}}\Big],\notag\\
h''(0)&=\frac{1}{3364 \,\mu_3^2(0)}\Big[-1682\, (\mu_3'(0))^2 \left(\log \frac{\mu_3^2(0)}{e^{70}}-2\right)+ 1682\,\mu_3(0)\,\mu_3''(0) \log \frac{\mu_3^2(0)}{e^{70}}\\&\qquad\qquad\qquad\quad+1740 \sqrt{2}\, \mu_3(0)\,\mu_3'(0)+150\, \mu_3^2(0) \left(\log \frac{\mu_3^3(0)}{e^{140}}+4\right)\Big]\,\notag
\end{align}
and since $\mu_3(0)=e^{40}$, $\mu_3'(0)=\frac{10\, \sqrt{2}\, e^{40}}{29}$, $\mu_3''(0)=\frac{25\, \left(e^{40}+8\, e^{50}\right)}{841 \,\left(e^{10}-1\right)}$, we finally obtain (cf.\ Fig.\ \ref{mdev2})
\begin{align}
h'(0)=0\,,\qquad 
h''(0)=-\frac{25 \left(4\, e^{10}-49\right)}{841 \left(e^{10}-1\right)}<0,
\end{align}
completing the proof of the non-rank-one-convexity of $W$.
\begin{figure}[h!]
	\begin{minipage}[h]{1\linewidth}
		\centering
		\includegraphics[scale=0.6,angle =0]{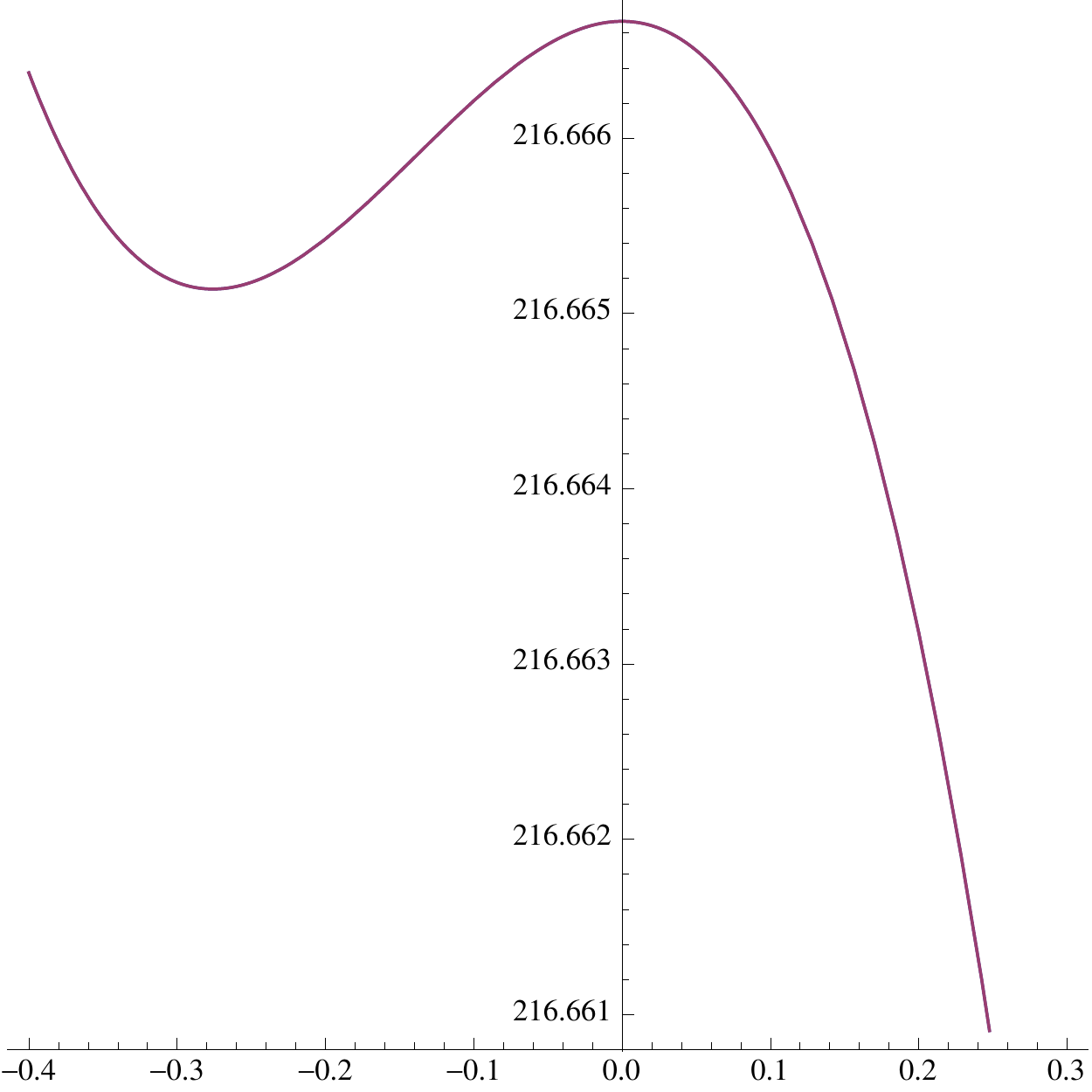}	
		\centering
		\caption{\footnotesize{Graphical representation of $h$, clearly showing the concave critical point at $t=0$.} }\label{mdev2}
	\end{minipage}
\end{figure}
\end{proof}

\subsection{Isochoric, tension-compression symmetric energies}
\label{section:isochoricTCsymmetricEnergies}
For $n\geq3$, not every isotropic energy on $\GLpn$ which is isochoric as well as tension-compression symmetric, i.e.\ satisfies $W(F) = W(F^{-1}=W(a\.F))$ for all $F\in\GLpn$ and all $a>0$, can be represented in terms of $\norm{\dev_n\log U}^2$ alone. An example of such a function which cannot be written in the form $W(F)=\Psi(\norm{\dev_n \log U}^2)$ is given by $W\col\mathrm{GL}^+(3)\to\mathbb{R}$ with
\[
	W(F) = [\det(\dev_3 \log U)]^2\,.
\]
It is straightforward to check that $W$ is isochoric and tension-compression symmetric. Furthermore, for
\begin{equation}\label{eq:isochoricRepresentationCounterexamples}
	U_1 = \matr{e&0& 0\\0&e&0\\0&0&e^{-2}}
	\qquad\text{ and }\qquad
	U_2 = \matr{e^{\sqrt{3}}&0& 0\\0&1&0\\0&0&e^{-\sqrt{3}}}\,,
\end{equation}
we find
\[
	W(U_1) = \left[\det \matr{1&0& 0\\0&1&0\\0&0&{-2}}\right]^2 = 4 \neq 0 = \left[\det \matr{{\sqrt{3}}&0& 0\\0&0&0\\0&0&{-\sqrt{3}}}\right]^2 = W(U_2)\,.
\]
However, since
\[
	\norm{\dev_3 \log U_1}^2 = \left\lVert\dev_3 \matr{1&0& 0\\0&1&0\\0&0&{-2}}\right\rVert^2 = 6 = \left\lVert\dev_3 \matr{{\sqrt{3}}&0& 0\\0&0&0\\0&0&{-\sqrt{3}}}\right\rVert^2 = \norm{\dev_3 \log U_2}^2\,,
\]
the equality $\Psi(\norm{\dev_3 \log U_1}^2)=\Psi(\norm{\dev_3 \log U_2}^2)$ must hold for all functions $\Psi\col[0,\infty)\to\R$. Therefore $W$ cannot be expressed in the form $W(F)=\Psi(\norm{\dev_3 \log U}^2)$.
\bigskip

Since our attempts to find a function $\Psi$ such that $F\mapsto\Psi(\norm{\dev_3 \log U}^2)$ is rank-one on $\GLp(3)$ turned out to be in vain, we considered the possibility that a tension-compression symmetric and isochoric elastic energy on $\GLp(3)$ cannot be rank-one convex in general. However, as the following example demonstrates, this assumption turned out to be false, showing that the non-ellipticity of \eqref{devf} for $n=3$ is a more particular drawback of the logarithmic formulation alone.

Consider the invariants $\Ihat_1,\Ihat_2,\Ihat_3\col\GLp(3)\to\mathbb{R}$ defined by
\begin{equation}\label{eq:isochoricInvariants}
	\Ihat_1(F) = \frac{\lambda_1^2}{\lambda_2\.\lambda_3}\,,\qquad\qquad
	\Ihat_2(F) = \frac{\lambda_1\.\lambda_2}{\lambda_3^2}\,,\qquad\qquad
	\Ihat_3(F) = \lambda_1\.\lambda_2\.\lambda_3
\end{equation}
for $F\in\GLp(3)$ with (ordered) singular values $\lambda_1\geq\lambda_2\geq\lambda_3$. Then $\Ihat_1$ and $\Ihat_2$ are isochoric, i.e.\ $\Ihat_i(a\.F)=\Ihat_i(F)$ for all $a>0$ and $i\in\{1,2\}$. Furthermore,\footnote{Note that the largest singular value of $F\inv$ is $\frac{1}{\lambda_3}$.}
\[
	\Ihat_1(F^{\inv}) = \frac{\lambda_3^{-2}}{\lambda_2^{-1}\.\lambda_1^{-1}} = \frac{\lambda_1\.\lambda_2}{\lambda_3^2} = \Ihat_2(F)
\]
and
\[
	\Ihat_2(F\inv) = \Ihat_1((F\inv)\inv) = \Ihat_1(F)\,.
\]

\begin{lemma}
	\label{lemma:multiplicativeInvariantsPolyconvexity}
	The functions $\Ihat_1$, $\Ihat_2$ and $\Ihat_3$ are polyconvex.
\end{lemma}
\begin{proof}
	We use the representations
	\begin{align}
	\widehat{I}_1(F) &= \frac{\lambda_1^2}{\lambda_2\.\lambda_3} = \frac{\lambda_1^3}{\lambda_1\.\lambda_2\.\lambda_3} = \frac{\norm{F}_2^3}{\det F}\,,\notag\\
	\widehat{I}_2(F) &= \frac{\lambda_1\.\lambda_2}{\lambda_3^2} = \lambda_1\.\lambda_2\.\lambda_3 \cdot \left(\frac{1}{\lambda_3}\right)^3 = (\det F)\cdot \norm{F\inv}_2^3 = w^*(F)\,,\\
	\widehat{I}_3(F) &= \lambda_1\.\lambda_2\.\lambda_3 = \det F\,,\notag
	\end{align}
	where $w^*$ denotes the \emph{Shield transformation}\footnote{The Shield transformation of $W^*$ of $W\col\GLpn\to\R^n$ is given by $W^*(F)=(\det F)\cdot W(F\inv)$.} \cite{shield1967inverse} of the function $w\col\GLp(3)\to\mathbb{R}$ with $w(F)=\norm{F}_2^3$. 
	
	The polyconvexity of $\widehat{I}_3$ is obvious, the proof of the polyconvexity of $\widehat{I}_1$ can be adapted from \cite[Lemma 3.2]{agn_ghiba2015exponentiated} and $\widehat{I}_2$ is polyconvex as the Shield transformation of the polyconvex mapping $w$ \cite{silhavy1997mechanics}.
\end{proof}

\begin{proposition}
\label{prop:isochoricTCsymmetricEnergy}
	The energy function $W\colon\GLp(3)\to\mathbb{R}$ with
	\[
	W(F) = \Ihat_1(F) + \Ihat_2(F) = \frac{\lambda_1^2}{\lambda_2\.\lambda_3} + \frac{\lambda_1\.\lambda_2}{\lambda_3^2}
	\]
	for all $F\in\GLp(3)$ with ordered singular values $\lambda_1\geq\lambda_2\geq\lambda_3$ is isochoric, tension-compression symmetric and polyconvex.
\end{proposition}
\begin{proof}
	The polyconvexity of $W$ follows directly from the polyconvexity of $\Ihat_1$ and $\Ihat_2$, see Lemma \ref{lemma:multiplicativeInvariantsPolyconvexity}. Similarly, $W$ is isochoric as the sum of the isochoric functions $\Ihat_1,\Ihat_2$. Furthermore,
	\[
	W(F\inv) = \Ihat_1(F\inv) + \Ihat_2(F\inv) = \Ihat_2(F) + \Ihat_1(F) = W(F)\,,
	\]
	thus $W$ is tension-compression symmetric as well.
\end{proof}

Since polyconvexity implies rank-one convexity, the energy function $W$ given in Proposition \ref{prop:isochoricTCsymmetricEnergy} is an example of a rank-one convex energy on $\GLp(3)$ which is isochoric and tension-compression symmetric.
However, $W$ cannot be expressed as a function of $\norm{\dev_3 \log U}^2$, since for $U_1,U_2$ as defined in \eqref{eq:isochoricRepresentationCounterexamples}, we find
\[
	W(U_1) = \frac{e^2}{e\cdot e^{-2}} + \frac{e\cdot e}{(e^{-2})^2} = e^3+e^6 \;\neq\; 2\.e^{3\sqrt{3}} = \frac{(e^{\sqrt{3}})^2}{1\cdot e^{-\sqrt{3}}} + \frac{e^{\sqrt{3}}\cdot1}{(e^{-\sqrt{3}})^2} = W(U_2)\,. \qedhere
\]

\section{The volumetric-isochoric split}\label{volisosect}

A function $W$ on $\GLpn$ is called \emph{volumetric-isochorically split} if it is of the form
\begin{align}
	W=\Wiso(F)+\Wvol(\det F)
\end{align}
with a function $\Wvol\col[0,\infty)\to\R$ and an objective, isotropic function $\Wiso\col\GLpz\to \mathbb{R}$ which is additionally isochoric, i.e.\ satisfies $\Wiso(a\,F)=\Wiso(F)$ for all $F\in\GLpz$ and all $a>0$.

In the two-dimensional case, every isochoric-volumetrically decoupled energy $W\col\GLp(2)\to \mathbb{R}_+$ can be written in the form \cite{agn_martin2015rank}
\begin{align}
\label{eq:isoVolDecoupled}
	W(F)=\Psi(\norm{\dev_2\log U}^2))+\Wvol(\det F)
\end{align}
with some function $\Psi\col\mathbb{R}_+\to \mathbb{R}_+$. According to Proposition 	\ref{prop:mainResultInTermsOfLogSquared} (cf.\ \cite[page 213]{Dacorogna08}), an energy $W$ of the form \eqref{eq:isoVolDecoupled} is rank-one convex for any convex function $\Wvol$ and any function $\Psi$ satisfying the inequality $2\,\eta\,\Psi^{\prime\prime}(\eta)+ (1-\sqrt{2\,\eta})\,\Psi^{\prime}(\eta)\geq 0$ for all $\eta\in(0,\infty)$. Hence, all the functions $W_{_{\rm eH}}\col\GLp(2)\to \R$ from the family of exponentiated Hencky type energies \cite{agn_neff2015exponentiatedI,agn_schroeder2017exponentiated,agn_nedjar2017finite}
 \begin{align}
 W_{_{\rm eH}}(F)=
 \dd\frac{\mu}{k}\,e^{k\,\norm{{\rm dev}_n\log U}^2}+\frac{\kappa}{2\widehat{k}}\,e^{\widehat{k}\,[(\log \det U)]^2}
 \end{align}
 are {rank-one convex} for $\mu>0, \kappa>0$, $k\geq\dd\frac{1}{4}$ and $\widehat{k}\dd\geq \tel8$.

In the three-dimensional case, not every objective, isotropic and isochoric energy function may be written as function of $\norm{\dev_3\log V}^2$. It is also known that there exist volumetric-isochorically decoupled energies which are rank-one convex, see for example Section \ref{section:isochoricTCsymmetricEnergies}.
However, we show that 
\begin{theorem}
	For $n\geq3$, there do not exist two-times continuously differentiable functions $\Psi\col[0,\infty)\to\mathbb{R}$ and $\Wvol\col[0,\infty)\to \mathbb{R}$ with $\Psi'(0)>0$ such that
	\begin{align}
	W(F)=\Psi(\norm{\dev_n\log V}^2)+\Wvol(\det F)
	\end{align}
	is rank-one convex on $\GLp(n)$.
\end{theorem}
\begin{remark}
	According to Remark \ref{remark:strictMonotonicityConditionLinearCompatibilityDev}, the condition $\Psi'(0)>0$ can be replaced by the requirement that $W$ is compatible with linear elasticity in the three-dimensional case.
\end{remark}
\begin{proof}
We prove the result only for $n=3$. First, we compute the second derivative for the volumetric part.
Since
\begin{align}
\det(F+H)&=
\det F+\langle H,\Cof F\rangle +\langle \Cof H,F\rangle+\det H
,\notag
\end{align}
we find
\begin{align}
D_F(\det F).H= \langle H,\Cof F\rangle = \det F\, \langle F^{-1}H,\id\rangle=\det F\, \tr(F^{-1}H)
\end{align}
as well as
\begin{align}
D_F^2(\det F).(H,H)=2\, \langle \Cof H,F\rangle \,.
\end{align}
With $H=\xi\otimes \eta$, using that $\Cof (\xi\otimes \eta)=0$, we thus find
\begin{align}
D_F^2\Wvol(\det F).(H,H)&=\Wvol^{\prime\prime}(\det F)\, [D_F(\det F).H]^2+\Wvol^{\prime}(\det F) \,\smash{\underbrace{D_F^2(\det F).(H,H)}_{=0}}\notag\\
&=\Wvol^{\prime\prime} (\det F)\,[\det F\, \langle F^{-1}H,\id\rangle]^2.
\end{align}
Regarding the isochoric part, we recall that
\begin{align}
D_F \Wiso(F).\xi\otimes\eta&=\Psi'(\norm{\dev_3\log V}^2)\cdot D_F(\norm{\dev_3\log V}^2).\xi\otimes\eta,\notag\\
D_F^2 \Wiso(F).(\xi\otimes\eta,\xi\otimes\eta)&=\Psi''(\norm{\dev_3\log V}^2)\cdot [D_F(\norm{\dev_3\log V}^2).\xi\otimes\eta]^2\\
&\qquad+\Psi'(\norm{\dev_3\log V}^2)\cdot D_F^2(\norm{\dev_3\log V}^2).(\xi\otimes\eta,\xi\otimes\eta).\notag
\end{align}
Therefore, the rank-one convexity condition for the total energy $W$ reads
\begin{align}
	0 &\leq \Psi''(\norm{\dev_3\log V}^2)\cdot [D_F(\norm{\dev_3\log V}^2).\xi\otimes\eta]^2\\
	&\quad+\Psi'(\norm{\dev_3\log V}^2)\cdot D_F^2(\norm{\dev_3\log V}^2).(\xi\otimes\eta,\xi\otimes\eta) + \Wvol^{\prime\prime}(\det F)\, \,(\det F)^2\, \langle F^{-1}(\xi\otimes\eta) ,\id\rangle^2\notag
\end{align}
or, equivalently (cf.\ \eqref{eq:firstDerivativeDevLog}),
\begin{align}\label{condalt}
	0 &\leq 4\,\Psi''(\norm{\dev_3\log V}^2)\cdot \langle (\dev_3\log V)\, \xi, F^{-T}\, \eta\rangle^2\notag\\
	&\quad+\Psi'(\norm{\dev_3\log V}^2)\cdot D_F^2(\norm{\dev_3\log V}^2).(\xi\otimes\eta,\xi\otimes\eta) + \Wvol^{\prime\prime}(\det F)\, \,(\det F)^2\, \langle \xi ,F^{-T}\eta\rangle^2\,.
\end{align}
In the following, we choose
\begin{align}
F_0=\begin{pmatrix}
1&0&0\\
0&e^{20}&0\\
0&0& e^{10}
\end{pmatrix},
\end{align}
and
\begin{align}
\xi=\begin{pmatrix}
\frac{\sqrt{3}}{2} \sin (\alpha ) \\
\frac{\sqrt{3}}{2} \cos (\alpha ) \\
\frac{1}{2} \\
\end{pmatrix},\qquad \eta=F_0^{T}\, \{\xi\times [(\dev_3\log V_0)\,\xi]\},
\end{align}
where $V_0=\sqrt{F_0\, F_0^T}=F_0$. Then obviously
\begin{align}\label{choice1}
\langle \xi ,F_0^{-T}\eta\rangle=0\qquad \text{and} \qquad \langle (\dev_3\log V_0)\, \xi, F_0^{-T}\, \eta\rangle=0.
\end{align}

Assume now that the energy $W$ is rank-one convex- Then, due to \eqref{condalt} and \eqref{choice1},
\begin{align}\label{condnou}
\Psi'(\norm{\dev_3\log V_0}^2)\cdot D_F^2(\norm{\dev_3\log V_0}^2)\Big|_{F=F_0}.(\xi\otimes\eta,\xi\otimes\eta)\geq 0\,. %
\end{align}
We will show that this inequality is violated. Again, we use the equality
\begin{align}
\hspace{-0.5cm}h''(0)&=
\frac{1}{4} D^2_F (\norm{\dev_3\log { F\,F^T}}^2)\Big|_{F=F_0}. (\xi \otimes \eta,\xi \otimes \eta),
\end{align}
where 
\begin{align}
h(t)&=\norm{\dev_3\log \sqrt{(F_0+t\,\xi\otimes\eta)(F_0+t\,\xi\otimes\eta)^T}}^2
=\frac{1}{12}\left[\log^2 \frac{\mu_1(t)}{\mu_2(t)}+\log^2 \frac{\mu_2(t)}{\mu_3(t)}+\log^2 \frac{\mu_3(t)}{\mu_1(t)}\right]\notag
\end{align}
for $t\in(-\eps,\eps)$, $\eps$ sufficiently small, and $\mu_i(t)$, $i=1,2,3$, denote the eigenvalues of $(F_0+t\,\xi\otimes\eta)(F_0+t\,\xi\otimes\eta)^T$.
\begin{figure}[h!]
	\begin{minipage}[h]{1\linewidth}
		\centering
		\includegraphics[scale=0.7]{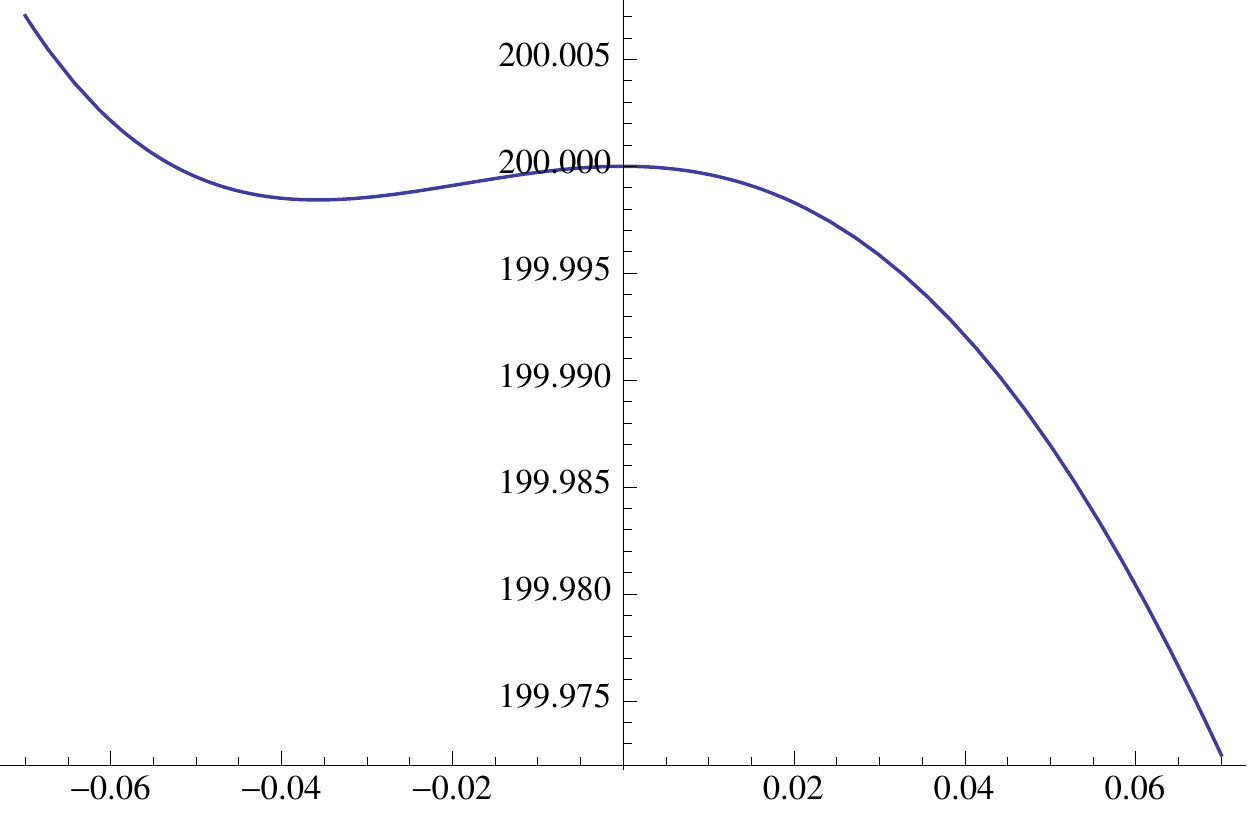} \hspace*{1cm}	
		\centering
		\caption{\footnotesize{The graph of the function $h$. At the critical point $t=0$, the function $h$ is concave, i.e.\ $h''(0)<0$, thus the rank-one convexity condition in \eqref{condalt} is violated.} }\label{grafvoliso3}
	\end{minipage}
\end{figure}

\noindent
After some lengthy computation, %
we find
\begin{align}
	h''(0)=\frac{75 \left(e^{40} \left(319-185 \sqrt{3}\right)+140 e^{20}+185 \sqrt{3}+301\right)}{16 \left(e^{40}-1\right)}\approx -6.70031<0\,,\notag
\end{align}
completing the proof; the graph of the mapping $t\mapsto h(t)$ is also shown in Fig.\ \ref{grafvoliso3}.
\end{proof}

\section*{Acknowledgment}
The work of Ionel-Dumitrel Ghiba was supported by a grant of the Romanian National Authority for Scientific Research and Innovation, CNCS-UEFISCDI, project number PN-III-P1-1.1-TE-2016-2314.

\begin{footnotesize}
	\printbibliography[heading=bibnumbered]
\end{footnotesize}

\appendix
\section{The derivative of $\norm{\log U}^2$}
\label{appendix:logDerivative}
There are multiple ways of computing the derivative of the mapping $F\mapsto \norm{\log U}^2$. Here, we discuss two ways of obtaining the derivative: a general formula for the trace of so-called \emph{primary matrix functions} and a method used by Vall\'ee \cite{vallee1978,vallee2008dual} (and indicated earlier by Richter \cite{richter1948}) for computing the Kirchhoff stress tensor corresponding to a hyperelastic material.

For any function $f\col\R^+\to\R$, we also denote by $f$ the corresponding \emph{primary matrix function}, which is uniquely defined by the equality
\begin{equation}
	f\big(Q^T\diag(\lambda_1,\dotsc,\lambda_n)\,Q\big) = Q^T\diag(f(\lambda_1),\dotsc,f(\lambda_n))\,Q
\end{equation}
for all $\lambda_i>0$ and all $Q\in\On$. If $f\col\R^+\to\R$ is differentiable, then \cite{agn_martin2015some}
\[
	D_S (\tr f(S)) = f'(S)\,,
\]
where $f'$ is interpreted as the primary matrix function corresponding to $f'\col\R^+\to\R$. In particular, since $\frac{d}{dt} \log^2(t) = \frac{2\.\log{t}}{t} \equalscolon w(t)$, we find
\begin{align*}
	D_B (\norm{\log B}^2) = D_B \tr(\log^2(B)) = w(B) = 2\.\log(B)\.B\inv = 2\.B\inv\.\log(B)\,.
\end{align*}

We can obtain the same result by applying Vall\'ee's general formula \cite{vallee2008dual}
\[
	D_X \Phi(\exp(X)).\Htilde = \iprod{(D\Phi)(\exp(X)),\, \exp(X)\cdot \Htilde}\,,
\]
which holds for any continuously differentiable isotropic function $\Phi\col\PSymn\to\R$, to the special case $\Phi(X)=\norm{\log X}^2$, which yields
\begin{equation}\label{eq:valleeFormula}
	\iprod{2\.X,\Htilde} = D_X (\norm{X}^2).\Htilde = D_X \Phi(\exp(X)).\Htilde = \iprod{(D\Phi)(\exp(X)),\, (\exp(X)\cdot \Htilde)}
\end{equation}
and thus, with $X=\log B$ and $\Htilde=B\inv H$,
\[
	\iprod{D_{\log B}\norm{\log B}^2,\, H} = \iprod{(D\Phi)(\log B),\, \exp(\log B) \cdot B\inv H} \overset{\eqref{eq:valleeFormula}}= \iprod{2\log (B)\.B\inv ,H} = \iprod{2B\inv\log (B) ,H}\,.
\]
We can now directly obtain the derivative with respect to $F$:
\begin{align*}
	\big( D_F \norm{\log U}^2 \big).H &= \frac14\, \big( D_F \norm{\log B}^2 \big).H\\
	&= \frac14\, \iprod{D_B(\norm{\log(B)}^2,\, D_F (B).H}\\
	&= \frac12\, \iprod{\log(B)\.B\inv,\, D_F (FF^T).H}\\
	&= \frac12\, \iprod{\log(B)\.B\inv,\, FH^T+HF^T}\\
	&= \frac12\, \iprod{B\inv\.\log(B),\, FH^T} + \frac12\, \iprod{\log(B)\.B\inv,\, HF^T}\\
	&= \iprod{F^{-T}F\inv\.\log(V),\, FH^T} + \iprod{\log(V)\.F^{-T}F\inv,\, HF^T}\\
	&= \iprod{F\inv\.\log(V),\, H^T} + \iprod{\log(V)\.F^{-T},\, H} = 2\.\iprod{\log(V)\.F^{-T},\, H}\,.
\end{align*}

\end{document}